\documentclass[reqno,a4paper,11pt]{amsart}
\usepackage[utf8]{inputenc}
\usepackage[final]{microtype}
\usepackage{amsmath,amssymb,amsthm}
\usepackage{mathrsfs}
\usepackage{mathtools}
\usepackage{commath}
\usepackage{thmtools}
\usepackage{thm-restate}
\usepackage{cases}
\usepackage{enumitem}
\setlist[enumerate,1]{label=(\roman*)}
\usepackage{xcolor}
\usepackage[pdfborderstyle={/S/U/W 0},unicode]{hyperref}
\hypersetup{
    colorlinks=true,
    linkcolor=blue!85!black,
    citecolor=blue!85!black,
    urlcolor=blue!85!black,
}
\usepackage{cleveref}
\usepackage{etoolbox}
\numberwithin{equation}{section}
\usepackage{thm-restate}
\usepackage{tikz}
\usepackage{tkz-euclide}
\usetikzlibrary{arrows.meta}
\crefname{subsection}{subsection}{subsections}

\makeatletter
\def\tkzDrawEllipse{\pgfutil@ifnextchar[{\tkz@DrawEllipse}{\tkz@DrawEllipse[]}}
\def\tkz@DrawEllipse[#1](#2,#3,#4,#5){%
\begingroup
\draw[#1](#2) ellipse [x radius=#3, y radius=#4,rotate=#5];
\endgroup
}
\makeatother

\ifdefined\thmcolor
\declaretheoremstyle[
  shaded={bgcolor=\thmcolor}
]{plain}
\else
\fi

\ifdefined\defcolor
\declaretheoremstyle[
  headfont=\normalfont\bfseries,
  bodyfont=\normalfont,
  shaded={bgcolor=\defcolor}
]{noital}
\else
\declaretheoremstyle[
  headfont=\normalfont\bfseries,
  bodyfont=\normalfont,
]{noital}
\fi

\declaretheorem[style=plain,numberwithin=section,name=Theorem]{theorem}
\declaretheorem[style=plain,sibling=theorem,name=Proposition]{proposition}
\declaretheorem[style=plain,sibling=theorem,name=Lemma]{lemma}
\declaretheorem[style=plain,sibling=theorem,name=Corollary]{corollary}
\declaretheorem[style=plain,sibling=theorem,name=Conjecture]{conjecture}
\declaretheorem[style=plain,sibling=theorem,name=Claim]{claim}
\declaretheorem[style=plain,sibling=theorem,name=Question]{question}
\declaretheorem[style=plain,sibling=theorem,name=Problem]{problem}
\declaretheorem[style=noital,sibling=theorem,name=Remark]{remark}

\newcommand{\defined}{\mathrel{\coloneqq}}
\newcommand{\defines}{\mathrel{\eqqcolon}}
\DeclarePairedDelimiter{\p}{\lparen}{\rparen}
\newcommand{\st}{\mathbin{\colon}}
\undef{\set}
\DeclarePairedDelimiter{\set}{\lbrace}{\rbrace}
\undef{\emptyset}
\newcommand{\emptyset}{\varnothing}
\DeclarePairedDelimiter{\card}{\lvert}{\rvert}
\newcommand{\from}{\colon}
\DeclarePairedDelimiter{\floor}{\lfloor}{\rfloor}
\DeclarePairedDelimiter{\ceil}{\lceil}{\rceil}
\DeclareMathOperator{\loglog}{log\,log}
\undef{\mod}
\newcommand{\mod}[1]{\ (\mathrm{mod}\ #1)}
\undef{\abs}
\DeclarePairedDelimiterX{\abs}[1]
  {\lvert}{\rvert}{\ifblank{#1}{\,\cdot\,}{#1}}
\undef{\norm}
\DeclarePairedDelimiterX{\norm}[1]
  {\lVert}{\rVert}{\ifblank{#1}{\,\cdot\,}{#1}}
\DeclarePairedDelimiterX{\inner}[2]
  {\langle}{\rangle}{\ifblank{#1}{\,\cdot\,}{#1},\ifblank{#2}{\,\cdot\,}{#2}}
\DeclareMathDelimiter{\given}
  {\mathbin}{symbols}{"6A}{largesymbols}{"0C}
\DeclareMathOperator{\Prob}{\mathbb{P}}
\DeclarePairedDelimiterXPP{\prob}[1]
  {\Prob}{\lparen}{\rparen}{}
  {\renewcommand{\given}{\nonscript\;\delimsize\vert\nonscript\;\mathopen{}}#1}
\DeclareMathOperator{\Expec}{\mathbb{E}}
\DeclarePairedDelimiterXPP{\expec}[1]
  {\Expec}{\lparen}{\rparen}{}
  {\renewcommand{\given}{\nonscript\;\delimsize\vert\nonscript\;\mathopen{}}#1}

\newcommand{\eps}{\varepsilon}

\let\SS\relax

\newcommand{\NN}{\mathbb{N}}

\newcommand{\RR}{\mathbb{R}}
\newcommand{\SS}{\mathbb{S}}

\newcommand{\ZZ}{\mathbb{Z}}
\newcommand{\cA}{\mathcal{A}}

\newcommand{\cE}{\mathcal{E}}

\newcommand{\cN}{\mathcal{N}}

\newcommand{\cP}{\mathcal{P}}
\newcommand{\cQ}{\mathcal{Q}}

\newcommand{\cS}{\mathcal{S}}

\def\cprime{\/{\mathsurround=0pt$'$}}

\DeclarePairedDelimiterX{\normstar}[1]
  {\lVert}{\rVert^*}{\ifblank{#1}{\,\cdot\,}{#1}}
\DeclarePairedDelimiterX{\normdag}[1]
  {\lVert}{\rVert^\dagger}{\ifblank{#1}{\,\cdot\,}{#1}}

\renewcommand{\d}{\mathop{}\!\mathrm{d}}
\newcommand{\dx}{\d x}

\newcommand{\dt}{\d t}
\newcommand{\dtheta}{\d \theta}

\usepackage[a4paper, margin=0.98in]{geometry}

\usepackage[backend=biber, style=ext-numeric, articlein=false, maxnames=99, giveninits=true, doi=false, isbn=false, url=false]{biblatex}
\bibliography{rev-littlewood-offord-arxiv-v1}

\begin{document}

\title{Double-jump phase transition for the reverse Littlewood--Offord problem}

\author[L. Hollom]{Lawrence Hollom}
\author[J. Portier]{Julien Portier}
\address{Department of Pure Mathematics and Mathematical Statistics, University of Cambridge, Cambridge, United Kingdom}
\email{lh569@cam.ac.uk}
\email{jp899@cam.ac.uk}

\author[V. Souza]{Victor Souza}
\address{Department of Computer Science and Technology, and Sidney Sussex College, University of Cambridge, Cambridge, United Kingdom}
\email{vss28@cam.ac.uk}

\begin{abstract}
Erdős conjectured in 1945 that for any unit vectors $v_1, \dotsc, v_n$ in $\mathbb{R}^2$ and signs $\varepsilon_1, \dotsc, \varepsilon_n$ taken independently and uniformly in $\{-1,1\}$, the random Rademacher sum $\sigma = \varepsilon_1 v_1 + \dotsb + \varepsilon_n v_n$ satisfies $\|\sigma\|_2 \leq 1$  with probability $\Omega(1/n)$.
While this conjecture is false for even $n$, Beck has proved that $\|\sigma\|_2 \leq \sqrt{2}$ always holds with probability $\Omega(1/n)$.
Recently, He, Ju\v{s}kevi\v{c}ius, Narayanan, and Spiro conjectured that the Erdős' conjecture holds when $n$ is odd.
We disprove this conjecture by exhibiting vectors $v_1, \dotsc, v_n$ for which $\|\sigma\|_2 \leq 1$ occurs with probability $O(1/n^{3/2})$.
On the other hand, an approximated version of their conjecture holds: we show that we always have $\|\sigma\|_2 \leq 1 + \delta$ with probability $\Omega_\delta(1/n)$, for all $\delta > 0$.
This shows that when $n$ is odd, the minimum probability that $\|\sigma\|_2 \leq r$
exhibits a double-jump phase transition at $r = 1$, as we can also show that $\|\sigma\|_2 \leq 1$ occurs with probability at least $\Omega((1/2+\mu)^n)$ for some $\mu > 0$.
Additionally, and using a different construction, we give a negative answer to a question of Beck and two other questions of He, Ju\v{s}kevi\v{c}ius, Narayanan, and Spiro, concerning the optimal constructions minimising the probability that $\|\sigma\|_2 \leq \sqrt{2}$.
We also make some progress on the higher dimensional versions of these questions.
\end{abstract}

\maketitle


\section{Introduction}
\label{sec:intro}

In their seminal work of 1943, Littlewood and Offord~\cite{Littlewood1943-ax} examined signed sums of complex numbers with unit norm and, in particular, the probability that these sums lie within an open ball of unit radius.
This research laid the groundwork for what is now known as Littlewood–Offord theory, which is broadly concerned with bounds on the probability that the random signed sum $\eps_1 v_1 + \dotsb + \eps_n v_n$ falls within a target set $S$, where $v_1,\dotsc,v_n$ are fixed vectors and $\eps_i$ are independent Rademacher random variables, that is, $\eps_i$ are uniformly distributed on $\set{-1, +1}$.

Littlewood and Offord, motived by the problem of estimating the number of zeros of random polynomials, considered the case where each $v_i$ is a complex number with norm at least $1$, showing that the probability that $\eps_1 v_1 + \dotsb + \eps_n v_n$ lies within any open ball of radius $1$ is at most $O(n^{-1/2} \log n)$.
While this result was sufficient for their purposes, the best possible result was found in 1945 by Erdős~\cite{Erdos1945-fu}, who used Sperner's theorem to show that the probability is at most $\binom{n}{\floor{n/2}}2^{-n}$, attained when $v_1 = \dotsb = v_n = 1$.

In his influential paper of 1945, Erdős~\cite{Erdos1945-fu} posed two conjectures.
The first of these asked for a generalisation of the problem of Littlewood and Offord to an arbitrary Hilbert space, and was resolved by Kleitman~\cite{Kleitman1970-os}.
Erdős' second conjecture is the following.

\begin{conjecture}[Erdős]
\label{conj:erdos}
Let $x_1, \dotsc, x_n$ be unit complex numbers.
Then the number of sums $\sum_{i=1}^n \eps_i x_i$ with $\eps_i \in \set{-1, +1}$ and $\abs[\big]{\sum_{i=1}^n \eps_i x_i} \leq 1$ is greater than $c 2^n / n$ for some absolute constant $c > 0$.
\end{conjecture}

Questions of this kind have been recently termed `reverse' Littlewood--Offord problems, as the goal is to show a lower bound on the number of signed sum $\eps_1 v_1 + \dotsb + \eps_n v_n$ that lie in a specified set, rather than an upper bound.

It turns out that \Cref{conj:erdos} is false as stated, which can be seen in $\RR^2$ by taking an odd number of copies of $(1,0)$ and of $(0,1)$.
Indeed, this forces all the sums to have norm at least $\sqrt{2}$.
This observation, which is attributed to Erdős, Sárközy, and Szemerédi by Beck~\cite{Beck1983-ef}, was also made by Carnielli and Carolino~\cite{Carnielli2011-mq}.
Both groups conjectured from this example that \Cref{conj:erdos} should hold if the radius $1$ is replaced with $\sqrt{2}$.
This corrected version of the conjecture of Erdős was proven by Beck~\cite{Beck1983-ef} in 1983, who moreover obtained the analogous result in every dimension.

\begin{theorem}[Beck]
\label{thm:beck}
For any $d \geq 1$, there is a constant $c_d >0$ depending only on $d$ such that the following holds.
Let $v_1, \dotsc, v_n$ be vectors in $\RR^d$ with $\norm{v_i}_2 \leq 1$ for each $1 \leq i \leq n$.
If $\eps_1, \dotsc, \eps_n$ are independent Rademacher random variables, then
\begin{align*}
    \prob[\big]{ \norm{\eps_1 v_1 + \cdots + \eps_n v_n}_2 \leq \sqrt{d} } \geq \frac{c_d}{n^{d/2}}.
\end{align*}
\end{theorem}

Recently, He, Ju\v{s}kevi\v{c}ius, Narayanan, and Spiro~\cite{He2024-cp} rediscovered this result for $d=2$ with an alternative proof.
While they note that the bound of $\sqrt{d}$ on the radius of the ball in \Cref{thm:beck} is optimal, this is not the end of the story for \Cref{conj:erdos}.
Indeed, when $d = 2$, the example that showed that a radius of $\sqrt{2}$ is required only works when $n$ is even.
Encouraged by the possibility that this is the only obstruction that prevents concentration inside the unit disk, they conjectured~\cite[Conjecture 4.1]{He2024-cp} that the original conjecture of Erdős holds when $n$ is odd.

\begin{conjecture}[He, Ju\v{s}kevi\v{c}ius, Narayanan, and Spiro]
\label{conj:HJNSodd}
There is a constant $c > 0$ such that, for every $n$ odd and unit vectors $v_1, \dotsc, v_n\in \RR^2$, we have
\begin{align*}
    \prob[\big]{ \norm{\eps_1 v_1 + \dotsb + \eps_n v_n}_2 \leq 1 } \geq \frac{c}{n}.
\end{align*}
\end{conjecture}

He, Ju\v{s}kevi\v{c}ius, Narayanan, and Spiro had already noted in~\cite{He2024-cp} that their pairing technique could be used to show $\prob{\norm{\eps_1 v_1 + \dotsb + \eps_n v_n}_2 \leq r } \geq \Omega(1/n)$ for some $r$ slightly smaller than $\sqrt{2}$, but new ideas are required to get close to $1$.
In our first result, we provide an approximate version of \Cref{conj:HJNSodd}, showing that for any $r$ arbitrarily close to $1$, a lower bound of order $1/n$ still holds.

\begin{theorem}
\label{thm:approximate}
For any $\delta > 0$ there is a constant $c_{\delta} >0$ such that, if $n$ is odd and $v_1, \dotsc, v_n \in \RR^2$ are unit vectors, then
\begin{align*}
    \prob[\big]{ \norm{\eps_1 v_1 + \cdots + \eps_n v_n}_2 \leq 1 + \delta } \geq \frac{c_{\delta}}{n}.
\end{align*}
\end{theorem}

While our proof develops on the pairing technique from~\cite{He2024-cp}, another important ingredient is the following vector balancing result of Swanepoel~\cite{Swanepoel2000-ha}, later reproved by Bárány, Ginzburg and V. S. Grinberg~\cite{Barany2013-vn}.

\begin{theorem}[Swanepoel]
\label{thm:alternating-sums}
Let $n$ be odd, and let $v_1, \dotsc, v_{n} \in \RR^2$ be unit vectors.
Then there exist signs $\eta_1, \dotsc, \eta_{n} \in \set{-1,1}$ such that
\begin{equation*}
    \norm[\Big]{\sum_{i=1}^{n} \eta_i v_i}_2 \leq 1.
\end{equation*}
\end{theorem}

In other words, this result shows that out of the $2^n$ possible signings $\eta_i$, at least one is such that $\sum_{i=1}^{n} \eta_i v_i$ falls inside a ball of radius $1$ centred at the origin.
Hence, this gives the weaker bound of $2^{-n}$ in place of $c/n$ for \Cref{conj:HJNSodd}.
In our next result, we provide an enhanced version of Swanepoel's result by showing that indeed there are exponentially many different signings with $\norm{\sum_{i=1}^{n} \eta_i v_i}_2 \leq 1$.

\begin{theorem}
\label{thm:lower-bound}
If $n \geq 1$ is odd and $v_1, \dotsc, v_n \in \RR^2$ are unit vectors, then
\begin{align*}
    \prob[\big]{ \norm{\eps_1 v_1 + \dotsb + \eps_n v_n}_2 \leq 1 } \geq \frac{1}{4} (0.525)^n.
\end{align*}
\end{theorem}

While \Cref{thm:lower-bound} provides an exponential improvement over \Cref{thm:alternating-sums}, this is still quite far from the bound of order $\Omega(1/n)$ in the original question of Erdős.
However, our next result shows that a bound of order $\Omega(1/n)$ cannot be attained, as \Cref{conj:HJNSodd} is false.

\begin{theorem}
\label{thm:radius1}
There is a constant $C > 0$ such that, for every $n$ odd, there exists unit vectors $v_1,\dotsc,v_n\in \RR^2$ such that
\begin{align}
\label{eq:radius1}
    \prob[\big]{ \norm{\eps_1 v_1 + \cdots + \eps_n v_n}_2 \leq 1 } \leq \frac{C}{n^{3/2}}.
\end{align}
\end{theorem}

In particular, Erdős' original conjecture from 1945 is not only false for even $n$ as previously noted, but it is false for odd $n$ as well.

The existence of constructions like those in \Cref{thm:radius1} is a delicate matter since the value of $\prob[\big]{ \norm{\eps_1 v_1 + \cdots + \eps_n v_n}_2 \leq 1 }$ has to be atypically small.
Indeed, $n \geq 2$ and $v_1, \dotsc, v_n$ are selected independently and uniformly at random from the circle $\SS^1 \subseteq \RR^2$, then
\begin{equation*}
    \Expec_{v_1,\dotsc,v_n \in \SS^1} \prob[\big]{ \norm{\eps_1 v_1 + \cdots + \eps_n v_n}_2 \leq 1 } = 1/(n+1),
\end{equation*}
a fact that traces back to the work of Rayleigh on `random flights'; see Bernardi~\cite{Bernardi2013-ql} for a modern and elementary proof.

\Cref{thm:radius1}, together with \Cref{thm:approximate}, showcases a surprising change of behaviour that occurs when considering the radius to be exactly $1$.
This \emph{double-jump} phase transition, reminiscent to the one that occurs with the size of the largest component of the Erdős-Rényi random graph~\cite{Erdos1960-vm}, illustrating the richness of phenomena exhibited by the reverse Littlewood--Offord problem.

After attending a seminar about our work, Gregory Sorkin~\cite{Sorkin25} found an alternative construction of unit vectors $v_1, \dotsc, v_n \in \RR^2$, for $n$ odd, such that
\begin{equation*}
     \prob[\big]{ \norm{\eps_1 v_1 + \cdots + \eps_n v_n}_2 \leq 1 } \leq 2^{-(n-1)/2}.
\end{equation*}
This not only show that our lower bound in \Cref{thm:lower-bound} is close to being sharp, but also demonstrate the acute contrast of behaviour at the radius of the double-jump, compared with other radii.
This also answers our \Cref{qu:critical2}.

For a set of vectors $V = \set{v_1, \dotsb, v_n} \subseteq \RR^d$, denote by $\sigma_V$ the random variable
\begin{equation*}
    \sigma_V \defined \eps_1 v_1 + \dotsb + \eps_n v_n,
\end{equation*}
where $\eps_1, \dotsc, \eps_n$ are independent Rademacher random variables.
Consider the quantity
\begin{equation*}
    F_{d,r}(n) \defined \inf_{V \in (\SS^{d-1})^n} \prob[\big]{\norm{\sigma_V}_2 \leq r}.
\end{equation*}
What we have seen above implies that the asymptotic behaviour of $F_{d,r}(n)$ may depend on the parity of $n$.
For instance, if $d = 2$ and $n$ is even, Beck's result (\Cref{thm:beck}) implies that
\begin{equation*}
    F_{2,r}(n) = \begin{cases}
        0 & \text{if $r < \sqrt{2}$}, \\
        \Theta_r(n^{-1}) & \text{if $r \geq \sqrt{2}$. }
    \end{cases}
\end{equation*}
On the other hand, if $n$ is odd, we now know from \Cref{thm:approximate}, \Cref{thm:lower-bound} and the construction of Sorkin that
\begin{equation*}
    F_{2,r}(n) = \begin{cases}
        0 & \text{if $r < 1$}, \\
        \Omega\p[\big]{0.525^n} \text{ and } O\p[\big]{(1/\sqrt{2})^{ n}} & \text{if $r = 1$}, \\
        \Theta_r(n^{-1}) & \text{if $r >1 $. }
    \end{cases}
\end{equation*}
Determining the precise order of magnitude of $F_{2,1}(n)$ when $n$ is odd remains an intriguing open problem, see further discussions in \Cref{sec:conclusion}.

Much less is known in higher dimensions.
By considering examples consisting of repeated orthogonal vectors, one can easily see that $F_{d,r}(n) = 0$ for all $r < \sqrt{d}$ when $n \equiv d \mod{2}$, and for all $r < \sqrt{d-1}$ when $n \not\equiv d \mod{2}$.
Beck's theorem shows that $F_{d,r}(n) = \Theta_r(n^{-d/2})$ for $r \geq \sqrt{d}$, regardless of the parity.
For $n \not\equiv d \mod{2}$, our proof of \Cref{thm:radius1} actually leads to the more general result below.

\begin{restatable}{theorem}{constsmallradius}
\label{thm:construction-small-radius}
For every $d \geq 1$, there is a constant $C_d > 0$ such that for every $n$ with $n \not\equiv d \pmod{2}$, there is a sequence of unit vectors $v_1, \dotsc, v_n \in \RR^d$ with
\begin{equation*}
    \prob[\big]{\norm{\eps_1 v_1 + \dotsb + \eps_n v_n}_2 \leq \sqrt{d - 1}} \leq  \frac{C_d}{n^{(d+1)/2}}.
\end{equation*}
\end{restatable}

Here again, Sorkin's construction can be generalized to higher dimensions, improving the upper bound above to $O_d((1/\sqrt{2})^{n})$.
On the other hand, no analogue of \Cref{thm:lower-bound} is known in higher dimensions.
In fact, even the much weaker bound implied by \Cref{thm:alternating-sums} is missing, see \Cref{qu:vector-balancing} and the discussion below.

To summarise, in any $d \geq 1$, Beck's theorem gives that if $n \equiv d \mod{2}$, we have
\begin{equation*}
    F_{d,r}(n) = \begin{cases}
        0 & \text{if $r < \sqrt{d}$}, \\
        \Theta_r(n^{-{d/2}}) & \text{if $r \geq \sqrt{d}$. }
    \end{cases}
\end{equation*}
However, when $n \not\equiv d \mod{2}$, we only know
\begin{equation*}
    F_{d,r}(n) = \begin{cases}
        0 & \text{if $r < \sqrt{d-1}$}, \\
        O_d((1/\sqrt{2})^{n}) & \text{if $r = \sqrt{d - 1}$. } \\
        \Theta_r(n^{-{d/2}}) & \text{if $r \geq \sqrt{d}$. }
    \end{cases}
\end{equation*}
In particular, there is no double-jump threshold in dimension $d = 1$ and there is not sufficient evidence to suggest that it occurs when $d \geq 3$.
Furthermore, the value of the critical radius
\begin{equation}
\label{eq:rcrit-def}
    r_c^{\ast}(d) \defined \inf\set[\Big]{r > 0 \st  \liminf_{\substack{n\to \infty\\n \not\equiv d \mod{2}}} F_{d,r}(n) > 0}
\end{equation}
is not known when $d \geq 3$, although it must be in the range $[\sqrt{d-1},\sqrt{d}]$.
We now pose a question in discrepancy theory that, if answered positively, would imply that $r_c^\ast(d) = \sqrt{d-1}$.

\begin{restatable}[Refined vector balancing]{question}{refinedbal}
\label{qu:vector-balancing}
Let $v_1, \dotsc, v_n \in \RR^d$ be unit vectors with $n \not\equiv d \mod{2}$.
Is it always the case that there are signs $\eta_1, \dotsc, \eta_{n} \in \set{-1,1}$ with
\begin{equation*}
    \norm[\Big]{\sum_{i=1}^{n} \eta_i v_i}_2 \leq \sqrt{d-1} \;\; ?
\end{equation*}
\end{restatable}

While this question asks for a straightforward generalisation of \Cref{thm:alternating-sums} from Swanepoel, it remains unsolved for any $d \geq 3$.

In contrast, the problem of determining
\begin{equation*}
    r_c(d) \defined \inf\set[\Big]{r > 0 \st  \liminf_{n\to \infty} F_{d,r}(n) > 0}
\end{equation*}
was already posed in 1963 as a special case of a problem of Dvoretzky~\cite{Dvoretzky1963-gh}, who was interested in arbitrary norm in place of $\norm{}_2$.
The fact that $r_c(d) = \sqrt{d}$ has been proved independently by many authors in the early 80's, such as Sevast{\cprime}yanov~\cite{Sevast-yanov1980-jf}, Spencer~\cite{Spencer1981-qa}, V. V. Grinberg (unpubished, see~\cite{Barany1981-mi}), Beck~\cite{Beck1983-ef}, and Bárány and V. S. Grinberg~(see \cite{Barany2008-ca} and \cite{Barany1981-mi}).

\subsection{Regarding optimal constructions}

Recall that \Cref{thm:beck} states that if $V$ consists of $n$ unit vectors in $\RR^d$, then we have
\begin{equation*}
    \prob[\big]{\norm{\sigma_V}_2 \leq \sqrt{d}} \geq c_d n^{-d/2}
\end{equation*}
for some universal constant $c_d > 0$, depending solely on the dimension $d$.
While the order of magnitude of $\Omega(n^{-d/2})$ in \Cref{thm:beck} is best possible, finding the best implicit constants and constructions that attain them remain an elusive problem.

In dimension $d = 1$, all the configurations are equivalent and it is easy to see that $c_1 = \sqrt{2/\pi}$ is asymptotically the best constant attainable.
While this is sharp when $n$ is even, the constant can be improved to $2\sqrt{2/\pi}$ if $n$ is restricted to be odd.
Note that Sárközy and Szemerédi (unpublished, see Beck~\cite{Beck1983-ef}) determined that $\sqrt{2/\pi}$ is also asymptotically the best constant in the more general case where we allow $\norm{v_i}_2 \leq 1$ rather than $\norm{v_i}_2 = 1$.

For $d = 2$, Beck~\cite{Beck1983-ef} asked whether the optimal constant $c_2$ was given by taking a number of copies of the vectors $(1,0)$ and $(0,1)$ as equal as possible.

\begin{question}[Beck]
\label{qu:beck}
Let $V \subseteq \RR^2$ consists of $n = 4k+2$ unit vectors.
Is it true that
\begin{equation*}
    \prob[\big]{\norm{\sigma_V}_2 \leq \sqrt{2}} \geq 4\binom{2k+1}{k}^2 /2^{n} \;\, ?
\end{equation*}
\end{question}

If true, the bound in \Cref{qu:beck} would be best possible, matching with the case when $V$ consists of $2k+1$ copies of $(1,0)$ and $2k+1$ copies of $(0,1)$.
He, Ju\v{s}kevi\v{c}ius, Narayanan, and Spiro went further and raised the following question.

\begin{question}[Question 4.2 in \cite{He2024-cp}]
\label{qu:HJNSFunctionf}
How does the function
\begin{align*}
    f(r) \defined \liminf_{n\to\infty} \, \inf_{V \in (\SS^1)^n} \, n \cdot \prob[\big]{\norm{\sigma_V} \leq r}
\end{align*}
behave? In particular, is $f(r)$ always a multiple of $4/\pi$?
\end{question}

They noted that the `in particular' part of this question would hold if the minimiser of the probability was always roughly $n/2$ copies of the vectors $(1, 0)$ and $(0, 1)$.
When $r = \sqrt{2}$, they conjectured the following stronger statement, matching \Cref{qu:beck}.

\begin{conjecture}[Conjecture 4.3 in \cite{He2024-cp}]
\label{conj:HJNSOptimalExample}
For all $n$ sufficiently large, there exists some $t = t(n) \leq n$ such that for any set $V \subseteq \RR^2$ of $n$ unit vectors satisfy
\begin{align*}
    \prob[\big]{\norm{\sigma_V}_2 \leq \sqrt{2}}
    \geq \prob[\big]{\norm{\sigma_{V'}}_2 \leq \sqrt{2}},
\end{align*}
where $V'$ consists of $t$ copies of $(1,0)$ and $n - t$ copies of $(0,1)$.
\end{conjecture}

Our next result gives a negative answer to \Cref{qu:beck}, to the second part of \Cref{qu:HJNSFunctionf} and disproves \Cref{conj:HJNSOptimalExample}.

\begin{theorem}
\label{thm:BetterConstant}
Let $u_1, u_2, u_3 \in \RR^2$ be the vertices of an equilateral triangle inscribed in the unit circle centred at the origin.
Let $n = 3k$ and let $V \subseteq \RR^2$ consist of $k$ copies of vector $u_1$, $k$ copies of vector $u_2$ and $k$ copies of vector $u_3$.
Then we have
\begin{equation*}
    \prob[\big]{ \norm{\sigma_V}_2 \leq \sqrt{2} }
    = \p[\big]{1+o(1)}\frac{2 \sqrt{3}}{\pi n}.
\end{equation*}
\end{theorem}

We only take $n = 3k$ in \Cref{thm:BetterConstant} for convenience as the same result holds for other values of $n$ as long as the number of copies of each of the vectors is roughly the same, see \Cref{rem:BetterConstant}.

\Cref{qu:beck}, \Cref{conj:HJNSOptimalExample} and the second part of \Cref{qu:HJNSFunctionf} are all predicated on the assumption that the optimal bound for \Cref{thm:beck} is attained when the vectors in $V$ are orthogonal, in which case we have
\begin{equation*}
    \prob[\big]{ \norm{\sigma_V}_2 \leq \sqrt{2} }
    = \p[\big]{1+o(1)}\frac{4}{\pi n}.
\end{equation*}
\Cref{thm:BetterConstant} implies that a construction based on an equilateral triangle outperforms the orthogonal in dimension two.
However, it is still far from clear whether this new construction is optimal, see \Cref{qu:mercedes}.

The situation seems to be much more complex in higher dimensions.
Indeed, it could be tempting to conjecture that the $(d+1)$-regular simplex in $d$ dimensions is always the optimal example.
However, in \Cref{thm:OrthoBetterThanSimplex} we show that this is not the case when the dimension is high enough.
We say that a set of vectors $V = \set{v_1, \dotsc, v_n} \subseteq \SS^{d-1}$ is \emph{of simplicial type} if there exists a regular $d$-simplex $W = \set{w_1, \dotsc, w_{d+1}}$ centred at the origin such that for every $i \in [n]$, we have $v_i = w_j$ for some $j \in [d+1]$.
Similarly we say that a set of vectors $V = \set{v_1, \dotsc, v_n} \subseteq \SS^{d-1}$ is \emph{of orthogonal type} if there exists an orthogonal basis $W = \set{w_1, \dotsc, w_{d}}$ such that for every $i \in [n]$, we have $v_i = w_j$ for some $j \in [d]$.

\begin{theorem}
\label{thm:OrthoBetterThanSimplex}
There exists $d_0 \geq 0$ such that for all $d \geq d_0$, there is $0 < \eps_d < 1$ with the following property.
For every sufficiently large $n$, there is a set of $n$ vectors $Y$ of orthogonal type such that every set of $n$ vectors $X$ of simplicial type satisfies
\begin{align*}
    \prob[\big]{\norm{\sigma_Y}_2 \leq \sqrt{d}} < \eps_d \, \prob[\big]{\norm{\sigma_X}_2 \leq \sqrt{d}}.
\end{align*}
In particular, one can take $\eps_d = 2^{-0.005d}$.
\end{theorem}

Nevertheless, there are always configurations being better than the orthogonal basis.

\begin{theorem}
\label{thm:MixedIsBetter}
There is a constant $0 < \delta < 1$ such that the following holds.
For every $d \geq 2$ and for sufficiently large $n$, there is set of $n$ vectors $Z$ such that for every set of $n$ vectors $Y$ of orthogonal type, we have
\begin{align*}
    \prob[\big]{\norm{\sigma_Z}_2 \leq \sqrt{d}} < \delta \, \prob[\big]{\norm{\sigma_Y}_2 \leq \sqrt{d}}.
\end{align*}
\end{theorem}

When $d \geq 3$, the construction in \Cref{thm:MixedIsBetter} is of \emph{mixed type}, obtained by gluing a low-dimensional simplex to an orthogonal frame.
It is again far from clear whether these constructions are the best possible.
Moreover, as it will be clear from further examples, the best constant in \Cref{thm:beck} may be sensitive to whether $n \equiv d \mod{2}$ or $n \not\equiv d \mod{2}$.

Even the problem of determining which set of vectors $X$ of orthogonal or of simplicial type minimise $\prob[\big]{\norm{\sigma_X}_2 \leq \sqrt{d}}$ is highly non-trivial.
Indeed, they are connected with certain results on counting the number of solutions to high dimensional quadratic Diophantine inequalities with certain parity restrictions.
We are able to fully determine the optimal set of vectors of orthogonal type in \Cref{subsec:orthogonal} and \Cref{app:orthogonal}.
For sets of vectors of simplicial type, we could merely show that they are less efficient than vectors of orthogonal type, as in \Cref{thm:OrthoBetterThanSimplex}.

In \Cref{sec:conclusion} we collect many problems that remain.

\subsection{Related work}

Before delving into the details of our proofs, we would like to highlight the richness of Littlewood--Offord theory, which encompasses several distinct types of problems.
These include `forward' problems, where one seeks upper bounds on the probability that $\sigma_V$ lands within a target set $S$, as were formulated the original problems of the theory.
In this subfield, recent progress has been made on the `polynomial Littlewood--Offord' problem by Meka, Nguyen and Vu~\cite{Meka2016-xf} and by Kwan and Sauermann~\cite{Kwan2023-kl}.
Another branch is concerned with `inverse' problems, which aim to exhibit structural properties of $V$ when $\sigma_V$ is likely to fall in $S$, as explored by Tao and Vu~\cite{Tao2009-aj}.
Lastly, `reverse' problems seek lower bounds on the probability that $\sigma_V$ falls within a target set $S$, with for instance Keller and Klein's resolution of Tomaszewski's conjecture~\cite{Keller2022-ve}, and the work of the first and second authors on Tomaszewski's counterpart problem~\cite{Hollom2023-jc}.

\subsection{Structure}

\Cref{sec:preliminaries} contains some preliminary results and estimates that we will use throughout.
We prove \Cref{thm:approximate} in \Cref{sec:approximate}, \Cref{thm:lower-bound} in \Cref{sec:exponential}.
In \Cref{sec:radius1}, we prove \Cref{thm:construction-small-radius}, which implies \Cref{sec:radius1}.
We deal with the optimal constructions in \Cref{sec:optimal-constructions}, where we show \Cref{thm:BetterConstant,thm:OrthoBetterThanSimplex,thm:MixedIsBetter}.
We conclude with several open problems in \Cref{sec:conclusion}.


\section{Preliminaries}
\label{sec:preliminaries}

One of the tools we will make use of is the following pairing result proved by He, Ju\v{s}kevi\v{c}ius, Narayanan, and Spiro~\cite{He2024-cp}.

\begin{proposition}
\label{prop:hjns}
Let $v_1, \dotsc, v_{2n} \in \RR^d$ be unit vectors and $r, \alpha > 0$ be real numbers such that
\begin{equation*}
    r^2 \geq \alpha + \sum_{i = 1}^{n}\norm{v_{2i} - v_{2i-1}}_2^2.
\end{equation*}
Then for independent Rademacher variables $\eps_1, \dotsc, \eps_{2n}$, we have
\begin{equation}
\label{eq:hjns}
    \prob[\big]{\norm{\eps_1 v_1 + \dotsb + \eps_{2n} v_{2n}}_2 \leq r} \geq \frac{c}{n^{d/2}},
\end{equation}
where $c = c_{d,r,\alpha} > 0$ is a constant that depends only on $d$, $r$ and $\alpha$.
\end{proposition}

The binary entropy function is defined as
\begin{equation*}
    H(p) = - p \log p - (1-p) \log(1-p).
\end{equation*}
Throughout this paper, we make constant use of the following form of Stirling's approximation, which follows from Robbins~\cite{Robbins1955-fv}.

\begin{proposition}
\label{fact:stirling}
The following approximation holds as $n\to\infty$.
\begin{equation*}
    n! = \p[\big]{1 + O(1/n)}\sqrt{2 \pi n} \p[\Big]{\frac{n}{e}}^n.
\end{equation*}
Moreover, if $\min\set{t, n - t} \to \infty$, then
\begin{equation}
\label{eq:binom-approx}
    \binom{n}{t} = \p[\big]{1 + O(1/\min\set{t, n - t})} \sqrt{\frac{n}{2\pi t(n-t)}} \; 2^{n H(t/n)}.
\end{equation}
\end{proposition}

The asymptotic behaviour of the sum of powers of binomial coefficients is also relevant for us.

\begin{proposition}
\label{prop:franel}
For all integers $q \geq 1$, we have, as $m \to \infty$,
\begin{equation*}
    \sum_{k=0}^{m} \binom{m}{k}^{q} = \p[\big]{1 + o(1)} \frac{2^{mq}}{\sqrt{q}} \p[\bigg]{\frac{2}{\pi m}}^{(q-1)/2}.
\end{equation*}
\end{proposition}

This result appears in the problem book of Pólya and Szegö~\cite[Part II, Problem 40]{Polya1925-tf}, see Farmer and Leth~\cite{Farmer2004-kc} for a stand-alone proof of \Cref{prop:franel}.
When $q = 3$, these quantities are know as the \emph{Franel numbers} (see~\cite[A000172]{oeis}).
In fact, for our applications, we will need to control the behaviour of a more general sum of products of binomial coefficients.

\begin{restatable}{proposition}{perturbedfranel}
\label{prop:perturbed-franel-2}
Fix integer $q \geq 1$ and let $x_1, \dotsc, x_q \in \ZZ$ and $m_1, \dotsc, m_q \in \NN$ be such that $m_i \equiv x_i \mod{2}$ for all $1 \leq i \leq q$.
Furthermore, write $X \defined \max\set{\abs{x_1}, \dotsc, \abs{x_q}}$ and $\ell \defined \min\set{m_1, \dotsc, m_q}$.
For any fixed $0 < \eps < 1/2$, if $X = o(\ell^{1/2 - \eps})$ as $\ell \to \infty$, then we have
\begin{equation*}
    \sum_{k \in \ZZ} \, \prod_{i=1}^{q} \binom{m_i}{(m_i+x_i)/2 + k}
    = \p[\Big]{1 + O_\eps\p[\Big]{\frac{X}{\ell^{1/2-\eps}}}} 2^{m_1 + \dotsb + m_q} \sqrt{\frac{\p{2/\pi}^{q-1}}{ \p[\big]{\prod_{i=1}^q m_i} \p[\big]{ \sum_{i=1}^q m_i^{-1}}}}.
\end{equation*}
\end{restatable}

The proof of \Cref{prop:perturbed-franel-2} is rather technical, and is therefore deferred to \Cref{app:binom}.

We will also need the following lower bound for the sum of products of binomial coefficients, which holds for any value of $\ell = \min\set{m_1, \dotsc, m_d}$ (as opposed to $\ell \rightarrow \infty$ for \Cref{prop:perturbed-franel-2}).

\begin{restatable}{proposition}{lbperturbedfranel}
\label{prop:lower-bound-perturbed-franel}
Fix integers $q, C \geq 1$ and $x_1,\dotsc,x_q\in\ZZ$, and write $X \defined \max\set{\abs{x_1}, \dotsc, \abs{x_q}}$.
Then there exists $n_0 = n_0(q,C,X)$ such that, for every $n \geq n_0$, the following holds.
If $m_1, \dotsc, m_q \geq 0$ are integers such that $m_1 + \dotsb + m_q = n$ and, for all $1 \leq i \leq q$ we have $m_i \geq \abs{x_i}$ and $m_i \equiv x_i \mod 2$, then
\begin{equation*}
    \sum_{k \in \ZZ} \, \prod_{i=1}^{q} \binom{m_i}{(m_i+x_i)/2 + k}
    \geq 2^{n-1} \p[\bigg]{\frac{2}{\pi n}}^{(q-1)/2} q^{(q-2)/2}.
\end{equation*}
Moreover, if there are distinct $i$ and $j$ such that $m_i \leq C$ and $m_j \leq C$, then in fact
\begin{equation*}
    \sum_{k \in \ZZ} \, \prod_{i=1}^{q} \binom{m_i}{(m_i+x_i)/2 + k}
    \geq 2^{n-1} \p[\bigg]{\frac{2}{\pi n}}^{(q-1)/2} q^{(q-2)/2} \, n^{1/8}.
\end{equation*}
\end{restatable}

We will also make use of the following inequality.

\begin{restatable}{proposition}{lbprodbinomial}
\label{prop:lower-bound-prod-binomial}
For every integers $q \geq 1$, $m_1, \dotsc, m_q \geq 0$ and $n$ such that $m_1 + \dotsb + m_q = n$, we have, as $n \to \infty$, that
\begin{align*}
    \frac{1}{2^n} \prod_{i=1}^q \binom{m_i}{(m_i + x_i)/2} \geq \p[\big]{1 + o(1)} \p[\bigg]{\frac{2q}{\pi n}}^{q/2}.
\end{align*}
\end{restatable}

The proofs of \Cref{prop:lower-bound-perturbed-franel,prop:lower-bound-prod-binomial} are rather technical and deferred to \Cref{app:binom}.


\section{An approximate version of Erdős' conjecture}
\label{sec:approximate}

Our goal in this section is to give a proof of \Cref{thm:approximate}, which amounts to showing that for any $\delta > 0$, there is a constant $c_\delta > 0$ such that any sequence of unit vectors $v_1, \dotsc, v_n \in \RR^2$ with $n$ odd satisfies $\prob{ \norm[\big]{\eps_1 v_1 + \cdots + \eps_n v_n}_2 \leq 1 + \delta } \geq c_\delta/n$.

We say that the vectors $v_1, \dotsc, v_{2k+1} \in \RR^2$ are in \emph{standard form} if we can write $v_i = (\cos \theta_i, \sin \theta_i)$ with $0 \leq \theta_1 \leq \dotsb \leq \theta_{2k+1} < \pi$.
Recall that $\sigma_V \defined \eps_1 v_1 + \cdots + \eps_n v_n$ and note that without loss of generality, we may assume that the vectors $v_i$ are given in standard form.
Indeed, the distribution of $\sigma_V$ does not change when we replace $v_i$ by $-v_i$, so all $v_i$ may be put in a half-circle.
We can then apply a rotation to all $v_i$ so that $\theta_1 = 0$ without changing the distribution of $\norm{\sigma_V}_2$.

A \emph{pairing} $\cP$ of a collection of vectors $v_1, \dotsc, v_n \in \RR^2$ is a set of disjoint pairs of indices in $[n]$.
Our first lemma shows that, for any collection of vectors in standard form, there is a pairing of all but one of the vectors such that the sum of the squared distances between paired vectors is not too large.

\begin{lemma}
\label{lem:optimal-pairing}
If vectors $v_1, \dotsc, v_{2k+1} \in \RR^2$ are given in standard form, then there is a pairing $\cP$ of them such that $\card{\cP} = k$,
\begin{equation*}
    \sum_{\set{i,j} \in \cP} \abs{\theta_i - \theta_j} \leq \pi/2,
    \quad \text{and} \quad \sum_{\set{i,j} \in \cP} \norm{v_i - v_j}_2^2 \leq 2.
\end{equation*}
\end{lemma}
\begin{proof}
Since $v_i = (\cos\theta_i, \sin\theta_i)$ are in standard form, we have
\begin{equation*}
    \sum_{i=1}^{k} \abs{\theta_{2i} - \theta_{2i-1}} + \sum_{i=1}^{k} \abs{\theta_{2i+1} - \theta_{2i}} = \theta_{2k-1} - \theta_1 \leq \pi,
\end{equation*}
so by taking $\cP$ to be either $\set[\big]{(2i-1,2i) \st i \in [k]}$ or $\set[\big]{(2i,2i+1) \st i \in [k]}$, we can ensure that $\sum_{\set{i,j} \in \cP} \abs{\theta_i - \theta_j} \leq \pi/2$.
We claim that the same pairing satisfies the second condition.
Indeed, note that $\norm{v_i - v_j}_2^2 = 2 - 2 \cos(\theta_i - \theta_j)$ and set $f \from [-\pi/2,\pi/2] \to \RR$ to be the function $f(x) = 2 - 2\cos x$.
Further note that $f(x) \leq 4\abs{x}/\pi$, so
\begin{equation*}
    \sum_{\set{i,j} \in \cP} \norm{v_i - v_j}_2^2 = \sum_{\set{i,j} \in \cP} f(\theta_i - \theta_j) \leq \frac{4}{\pi}\sum_{\set{i,j} \in \cP}\abs{\theta_i - \theta_j} \leq 2. \qedhere
\end{equation*}
\end{proof}

We remark that \Cref{lem:optimal-pairing} cannot be improved, as can be seen by taking the vectors $v_1 = \dotsb = v_{2k-1} = (1,0)$, $v_{2k} = (0,1)$, and $v_{2k+1} = (-\cos \theta, \sin \theta)$ for some arbitrarily small $\theta > 0$.
Indeed, any maximal pairing $\cP$ of these vectors must pair some two distinct vectors with each other, which leads to both conditions being sharp.

For a pairing $\cP$, we define
\begin{equation*}
    E_1(\cP) \defined \sum_{\set{i,j} \in \cP} \abs{\theta_i - \theta_j},
    \quad \text{and} \quad E_2(\cP) \defined \sum_{\set{i,j} \in \cP} \norm{v_i - v_j}_2^2.
\end{equation*}
Therefore \Cref{lem:optimal-pairing} shows that any $2k+1$ vectors in standard form admit a pairing $\cP$ with $\card{\cP} = k$, $E_1(\cP) \leq \pi/2$ and $E_2(\cP) \leq 2$.
Our next goal is to show that if $E_2(\cP)$ is large, then there must be a pair in $\cP$ which contributes substantially to $E_2(\cP)$.
To establish this, we will use the following fact about convex functions.

\begin{lemma}
\label{lem:convex}
Let $\varphi \from [0,\Delta] \to \RR$ be convex, increasing and such that $\varphi(0) = 0$.
If $0 \leq x_1, \dotsb, x_k \leq \Delta$ are such that $x_1 + \dotsb + x_k \leq S$, then we have
\begin{equation}
\label{eq:convex-ineq}
    \varphi(x_1) + \dotsb + \varphi(x_k) \leq \p[\big]{\floor{S/\Delta} + 1} \varphi(\Delta).
\end{equation}
\end{lemma}
\begin{proof}
By definition of convexity, we have $\varphi(x_i - t) + \varphi(x_j + t) \geq \varphi(x_i) + \varphi(x_j)$ for all $t \geq 0$ such that $x_i - t \geq 0$ and $x_j + t \leq \Delta$.
Therefore the maximum of $\varphi(x_1) + \dotsb + \varphi(x_k)$ with the constraints that $0 \leq x_1, \dotsb, x_k \leq \Delta$ and $x_1 + \dotsb + x_k \leq S$ is attained when all $x_i$, except for maybe one, are equal to $0$ or $\Delta$. This gives
\begin{equation*}
    \varphi(x_1) + \dotsb + \varphi(x_k) \leq \floor{S/\Delta} \varphi(\Delta) + \varphi(S -
 \Delta\floor{S/\Delta}),
\end{equation*}
which implies \eqref{eq:convex-ineq} since $\varphi$ is increasing.
\end{proof}

We now show that every pairing $\cP$ with large $E_2(\cP)$ contains a pair $\set{i,j} \in \cP$ for which $\norm{v_i - v_j}_2^2$ is also large.

\begin{lemma}
\label{lem:strong-pair}
If $\cP$ is a paring of a collection $v_1, \dotsc, v_{2k+1}$ of vectors in standard form with $E_1(\cP) \leq \pi/2$ and $E_2(\cP) \geq \beta$, then there is $\set{i,j} \in \cP$ such that $\norm{v_i - v_j}_2^2 \geq \beta^2/10$.
\end{lemma}
\begin{proof}
Recall that $E_1(\cP) = \sum_{\set{i,j} \in \cP} \abs{\theta_i - \theta_j} \leq \pi/2$ and that
\begin{equation*}
    E_2(\cP) = \sum_{\set{i,j} \in \cP} \norm{v_i - v_j}_2^2 = \sum_{\set{i,j} \in \cP} f(\abs{\theta_i - \theta_j}),
\end{equation*}
where $f \from [0,\pi/2] \to \RR$ is defined as $f(x) \defined 2 - 2\cos x$.
It is straightforward that $f$ is convex, increasing, and satisfies $f(0) = 0$.
If $f(\abs{\theta_i - \theta_j}) \leq \delta$ for all $\set{i,j} \in \cP$, then $\abs{\theta_i - \theta_j} \leq f^{-1}(\delta) \defines \Delta$, where $f^{-1} \from [0,2] \to \RR$ is defined by $f^{-1}(x) \defined \arccos(1 - x/2)$.
Applying \Cref{lem:convex}, we obtain
\begin{equation*}
    E_2(\cP) \leq \p[\big]{\floor{(\pi/2)/\Delta} + 1} f(\Delta)
    \leq \p[\big]{\pi/(2f^{-1}(\delta)) + 1} \delta
\end{equation*}
Now define the function $g(x) = (\pi/(2f^{-1}(x)) + 1)x$ and note\footnote{To see this, note that $\arccos(1-t) \leq \pi/2$ for $0 \leq t \leq 1$, so $\arccos(1-x^2/20) \leq \pi/2$ for $x \leq \pi/2 < \sqrt{20}$.
So in this range, we have $g(x^2/10) \leq  \pi x^2/(10\arccos(1-x^2/20))$, so it suffices to show that $\pi x/ 10 \leq \arccos(1-x^2/20)$.
Alternatively, we must show that $\cos(\pi x/10) \geq 1 - x^2/20$, which follows from the inequality $\cos x \geq 1 - x^2/2$.} that $g(x^2/10) < x$ for all $0 < x < \pi/2$.
Therefore, if $E_2(\cP) \geq \beta$ and $f(\abs{\theta_i - \theta_j}) \leq \beta^2/10$ for all $\set{i,j} \in \cP$, we would have $\beta \leq E_2(\cP) \leq g(\beta^2/10) < \beta$, a clear contradiction.
\end{proof}

An immediate corollary of \Cref{lem:strong-pair} is the following.

\begin{corollary}
\label{cor:delta-pairing}
For every collection of unit vectors $v_1, \dotsc, v_{2k+1} \in \RR^2$ in standard form and any $\delta > 0$, there is a pairing $\cP$ of them with $E_2(\cP) \leq \delta$ and $\card{\cP} \geq k - \floor{10/\delta}$.
\end{corollary}
\begin{proof}
Apply \Cref{lem:optimal-pairing} to find an initial pairing $\cP_0$ with $\card{\cP_0} = k$ and $E_2(\cP_0) \leq 2$.
We now greedily remove pairs from $\cP_0$ until we get $E_2(\cP) \leq \delta$.
Indeed, inductively define $\cP_{n+1}$ by removing from $\cP_n$ a pair $\set{i,j}$ that has maximal value of $\norm{v_i - v_j}_2^2$.

Let $\beta_n \defined E_2(\cP_n)$.
We know by \Cref{lem:strong-pair} that $\beta_{n+1}\leq \beta_n - \beta_n^2/10$.
The function $f(x) = x - x^2/10$ is increasing for $0\leq x \leq 2$, and so if we can find a sequence of numbers $\gamma_n$ with $\gamma_{n+1}\geq \gamma_n - \gamma_n^2/10$ for all $n$ and $\gamma_0 = 2$, then we will know inductively that $\beta_{n+1} \leq f(\beta_n) \leq f(\gamma_n) \leq \gamma_{n+1}$ for all $n$.
In particular, if $\gamma_n \leq \delta$, then $E_2(\cP_n) =\beta_n \leq \delta$.

It is easy to check that the sequence $\gamma_n = 10/(n+5)$ has the required properties, and that if $n \geq 10/\delta - 5$ then $\gamma_n \leq \delta$.
The desired result follows by taking $\cP = \cP_{\ceil{10/\delta - 5}}$.
\end{proof}

We now deduce \Cref{thm:approximate} by combining \Cref{prop:hjns}, \Cref{thm:alternating-sums} and \Cref{cor:delta-pairing}.

\begin{proof}[Proof of \Cref{thm:approximate}]
Let $n = 2k + 1$ and consider unit vectors $v_1, \dotsc, v_n \in \RR^2$.
As explained earlier, we may assume without loss of generality that the vectors $v_1, \dotsc, v_n$ are in standard form. Moreover, we may assume that $n \geq 80/\delta^2$, as otherwise the result is obvious from \Cref{thm:alternating-sums}.
By \Cref{cor:delta-pairing}, we have a partial pairing $\cP$ satisfying $E_2(\cP) \leq \delta^2/2$ and $\card{\cP} \geq k - 20/\delta^2$.
We now partition the index set $[n]$ into $A \cup B$, where $A = \cup \cP$ consists of all indices $i$ which belongs to a pair in $\cP$, and $B = [n] \setminus A$.
Note that $\card{A} \geq 2k - 40/\delta^2$, while $\card{B}$ is an odd number satisfying $\card{B} \leq 40/\delta^2 + 1$.
By \Cref{thm:alternating-sums}, there is a set of signs $\set{\eta_i}_{i \in B}$ with $\eta_i = \pm 1$, such that $\norm[\big]{\sum_{i \in B} \eta_i v_i}_2 \leq 1$.
This implies the bound
\begin{equation*}
    \prob[\bigg]{ \norm[\Big]{\sum_{i \in B}\eps_i v_i}_2 \leq 1} \geq \frac{1}{2^{\card{B}}}.
\end{equation*}
Moreover, applying \Cref{prop:hjns} with $r= \delta$ and $\alpha = \delta^2/2$ gives that
\begin{equation*}
    \prob[\bigg]{\norm[\Big]{\sum_{i \in A} \eps_i v_i}_2 \leq \delta} \geq \frac{c_\delta}{\card{A}},
\end{equation*}
for some constant $c_\delta > 0$ that depends only on $\delta$.
By the triangle inequality, we deduce from the above equations that
\begin{align*}
    \prob[\bigg]{ \norm[\Big]{\sum_{i \in [n]}\eps_i v_i}_2 \leq 1 + \delta } &\geq \prob[\bigg]{ \norm[\Big]{\sum_{i \in A}\eps_i v_i}_2 \leq \delta} \prob[\bigg]{ \norm[\Big]{\sum_{i \in B}\eps_i v_i}_2 \leq 1} \\
    &\geq \frac{c_\delta}{\card{A}} \cdot \frac{1}{2^{\card{B}}}
    \geq \frac{c_\delta}{(n - 40/\delta^2) 2^{40/\delta^2+1}} \geq \frac{c'_\delta}{n},
\end{align*}
for some constant $c'_\delta > 0$ depending only on $\delta$, as claimed.
\end{proof}


\section{Refined vector balancing in the plane}
\label{sec:exponential}

In this section, our aim is to prove \Cref{thm:lower-bound}, which we may recall states that in the plane, for any unit vectors $v_1, \dotsc, v_n \in \RR^2$ when $n$ is odd, we have the following exponential lower bound for the probability that the random signed sum of the vectors lies in the unit disc:
\begin{align*}
    \prob[\big]{ \norm{\eps_1 v_1 + \dotsb + \eps_n v_n}_2 \leq 1 } \geq \frac{1}{4} (0.525)^n.
\end{align*}
We now proceed to the proof.

\begin{proof}[Proof of \Cref{thm:lower-bound}]
We start with unit vectors $v_1, \dotsc, v_n \in \RR^2$ and let $P$ be the convex hull of $\set{\pm v_1, \dotsc, \pm v_n}$.
We may assume, by replacing a vector $v_i$ by $-v_i$, that the vectors $v_1, \dotsc, v_n, -v_1, \dotsc, -v_n$ occur in this order as the vertices of $P$.
Define  $u \defined \sum_{i=1}^{n}(-1)^{i-1} v_i$.
As $n$ is odd, this corresponds to adding up every second vertex of the polygon $P$.
We will now show that $\norm{u}_2 \leq 1$.
Indeed, assume that $u \neq 0$, let $U$ be the linear span of the vector $u$ and let $b$ and $-b$ be the points in which the line $U$ intersects with the boundary of $P$.
See \Cref{fig:balancing} for the setup.
By relabelling\footnote{Note that a cyclic relabelling may not be enough, but this is always possible if we allow a change in the orientation of the labelling.} the vertices of $P$, we may assume that $b$ belongs to the edge $[-v_n,v_1]$.
Writing $a_i = v_{i+1} - v_i$ for $1 \leq i \leq n-1$ and $a_n = -v_1 - v_n$, we then have $2u = \sum_{i=1}^n (-1)^{i} a_i$.
Let $\pi_U$ be the oblique projection $\RR^2 \to U$ that sends $v_1$ to $b$.
Then we have
\begin{equation*}
    \norm{2u}_2 = \norm[\Big]{\sum_{i=1}^n (-1)^{i}\pi_U(a_i)}_2
    \leq \sum_{i=1}^n \norm[\big]{\pi_U(a_i)}_2
    \leq \norm{2b}_2,
\end{equation*}
see \Cref{fig:balancing} (middle) for an explanation of the last inequality.
We conclude that $\norm{u}_2 \leq \norm{b}_2 \leq 1$.
So far, we have closely followed the proof of \Cref{thm:alternating-sums} due to Bárány, Ginzburg and V. S. Grinberg~\cite{Barany2013-vn}.
By a rotation, we may assume that the line $U$ aligns with $x$-axis, so $u = (\beta, 0)$ for some $0 \leq \beta \leq 1$.
Reflect all vectors $v_i$ so they have non-negative $x$-coordinate and relabel them so they are $v_1, \dotsc, v_n$ counter-clockwise, as in \Cref{fig:balancing} (right).
After relabelling, we have that $\sum_{i=1}^{n}(-1)^{i-1} v_i$ is equal to $\pm u$.
Replace $u$ by $-u$ if needed so that this sum is equal to $u$.

\begin{figure}[!ht]
\centering
\begin{tikzpicture}[scale=1.92]
\definecolor{col_a}{RGB}{195, 66, 63};
\tkzDefPoints{0/-1.4/BT,0/1.4/BB}
\tkzDrawSegment[opacity=0](BT,BB)
\tkzDefPoints{0/0/O, 1/0/R}
\tkzDefPointsBy[rotation=center O angle -64](R){V1}
\tkzDefPointsBy[rotation=center O angle -40](R){V2}
\tkzDefPointsBy[rotation=center O angle -30](R){V3}
\tkzDefPointsBy[rotation=center O angle -20](R){V4}
\tkzDefPointsBy[rotation=center O angle -10](R){V5}
\tkzDefPointsBy[rotation=center O angle 20](R){V6}
\tkzDefPointsBy[rotation=center O angle 45](R){V7}
\tkzDefPointsBy[rotation=center O angle 65](R){V8}
\tkzDefPointsBy[rotation=center O angle 85](R){V9}
\tkzDefPointsBy[rotation=center O angle 180-64](R){U1}
\tkzDefPointsBy[rotation=center O angle 180-40](R){U2}
\tkzDefPointsBy[rotation=center O angle 180-30](R){U3}
\tkzDefPointsBy[rotation=center O angle 180-20](R){U4}
\tkzDefPointsBy[rotation=center O angle 180-10](R){U5}
\tkzDefPointsBy[rotation=center O angle 180+20](R){U6}
\tkzDefPointsBy[rotation=center O angle 180+45](R){U7}
\tkzDefPointsBy[rotation=center O angle 180+65](R){U8}
\tkzDefPointsBy[rotation=center O angle 180+85](R){U9}
\tkzDefPoints{0.0154188810884/-0.1326609220379/U}
\tkzInterLL(V1,U9)(O,U)\tkzGetPoint{B}
\tkzInterLL(V9,U1)(O,U)\tkzGetPoint{B'}
\tkzDrawCircle[line width=0.5pt, gray](O,R)
\tkzDrawSegments[black, line width=0.75pt](V1,V2 V2,V3 V3,V4 V4,V5 V5,V6 V6,V7 V7,V8 V8,V9 V9,U1 U1,U2 U2,U3 U3,U4 U4,U5 U5,U6 U6,U7 U7,U8 U8,U9 U9,V1)
\tkzDrawSegments[col_a, line width=0.8pt, arrows=-Stealth](O,V1 O,V3 O,V5 O,V7 O,V9 O,U2 O,U4 O,U6 O,U8)
\tkzDrawPoints[size=2.5pt, black](V1,V2,V3,V4,V5,V6,V7,V8,V9)
\tkzDrawPoints[size=2.5pt, black](U1,U2,U3,U4,U5,U6,U7,U8,U9)
\tkzDrawPoints[size=2.5pt, black](O,U)
\tkzLabelLine[pos=1.2](O,V1){$v_1$}
\tkzLabelLine[pos=1.2](O,V2){$v_2$}
\tkzLabelLine[pos=1.2](O,V3){$v_3$}
\tkzLabelLine[pos=1.2](O,V4){$v_4$}
\tkzLabelLine[pos=1.2](O,V5){$v_5$}
\tkzLabelLine[pos=1.2](O,V6){$v_6$}
\tkzLabelLine[pos=1.2](O,V7){$v_7$}
\tkzLabelLine[pos=1.2](O,V8){$v_8$}
\tkzLabelLine[pos=1.2](O,V9){$v_9$}
\tkzLabelLine[pos=1.2](O,U1){$-v_1$}
\tkzLabelLine[pos=1.25](O,U2){$-v_2$}
\tkzLabelLine[pos=1.25](O,U3){$-v_3$}
\tkzLabelLine[pos=1.25](O,U4){$-v_4$}
\tkzLabelLine[pos=1.22](O,U5){$-v_5$}
\tkzLabelLine[pos=1.21](O,U6){$-v_6$}
\tkzLabelLine[pos=1.2](O,U7){$-v_7$}
\tkzLabelLine[pos=1.2](O,U8){$-v_8$}
\tkzLabelLine[pos=1.18](O,U9){$-v_9$}
\tkzLabelPoint[xshift=-3pt,yshift=-1.5pt](U){$u$}
\end{tikzpicture} 
\begin{tikzpicture}[scale=1.92]
\definecolor{col_a}{RGB}{195, 66, 63};
\tkzDefPoints{0/-1.4/BT,0/1.4/BB}
\tkzDrawSegment[opacity=0](BT,BB)
\tkzDefPoints{0/0/O, 1/0/R}
\tkzDefPointsBy[rotation=center O angle -64](R){V1}
\tkzDefPointsBy[rotation=center O angle -40](R){V2}
\tkzDefPointsBy[rotation=center O angle -30](R){V3}
\tkzDefPointsBy[rotation=center O angle -20](R){V4}
\tkzDefPointsBy[rotation=center O angle -10](R){V5}
\tkzDefPointsBy[rotation=center O angle 20](R){V6}
\tkzDefPointsBy[rotation=center O angle 45](R){V7}
\tkzDefPointsBy[rotation=center O angle 65](R){V8}
\tkzDefPointsBy[rotation=center O angle 85](R){V9}
\tkzDefPointsBy[rotation=center O angle 180-64](R){U1}
\tkzDefPointsBy[rotation=center O angle 180-40](R){U2}
\tkzDefPointsBy[rotation=center O angle 180-30](R){U3}
\tkzDefPointsBy[rotation=center O angle 180-20](R){U4}
\tkzDefPointsBy[rotation=center O angle 180-10](R){U5}
\tkzDefPointsBy[rotation=center O angle 180+20](R){U6}
\tkzDefPointsBy[rotation=center O angle 180+45](R){U7}
\tkzDefPointsBy[rotation=center O angle 180+65](R){U8}
\tkzDefPointsBy[rotation=center O angle 180+85](R){U9}
\tkzDefPoints{0.01541888/-0.13266092/U}
\tkzInterLL(V1,U9)(O,U)\tkzGetPoint{B}
\tkzInterLL(V9,U1)(O,U)\tkzGetPoint{B'}
\tkzDefMidPoint(V1,V2) \tkzGetPoint{A1}
\tkzDefMidPoint(V2,V3) \tkzGetPoint{A2}
\tkzDefMidPoint(V3,V4) \tkzGetPoint{A3}
\tkzDefMidPoint(V4,V5) \tkzGetPoint{A4}
\tkzDefMidPoint(V5,V6) \tkzGetPoint{A5}
\tkzDefMidPoint(V6,V7) \tkzGetPoint{A6}
\tkzDefMidPoint(V7,V8) \tkzGetPoint{A7}
\tkzDefMidPoint(V8,V9) \tkzGetPoint{A8}
\tkzDefMidPoint(V9,U1) \tkzGetPoint{A9}
\tkzDefPointBy[translation=from V1 to B](V2) \tkzGetPoint{Y2}
\tkzDefPointBy[translation=from V1 to B](V3) \tkzGetPoint{Y3}
\tkzDefPointBy[translation=from V1 to B](V4) \tkzGetPoint{Y4}
\tkzDefPointBy[translation=from V1 to B](V5) \tkzGetPoint{Y5}
\tkzDefPointBy[translation=from V1 to B](V6) \tkzGetPoint{Y6}
\tkzDefPointBy[translation=from V1 to B](V7) \tkzGetPoint{Y7}
\tkzDefPointBy[translation=from V1 to B](V8) \tkzGetPoint{Y8}
\tkzDefPointBy[projection=onto B--B'](Y2) \tkzGetPoint{X2}
\tkzDefPointBy[projection=onto B--B'](Y3) \tkzGetPoint{X3}
\tkzDefPointBy[projection=onto B--B'](Y4) \tkzGetPoint{X4}
\tkzDefPointBy[projection=onto B--B'](Y5) \tkzGetPoint{X5}
\tkzDefPointBy[projection=onto B--B'](Y6) \tkzGetPoint{X6}
\tkzDefPointBy[projection=onto B--B'](Y7) \tkzGetPoint{X7}
\tkzDefPointBy[projection=onto B--B'](Y8) \tkzGetPoint{X8}
\tkzDrawCircle[line width=0.25pt, gray](O,R)
\tkzDrawLines[line width=0.5pt, gray, dashed, add=0.2 and 0.2](B,B')
\tkzDrawSegments[black, line width=0.75pt](V1,V2 V2,V3 V3,V4 V4,V5 V5,V6 V6,V7 V7,V8 V8,V9 V9,U1 U1,U2 U2,U3 U3,U4 U4,U5 U5,U6 U6,U7 U7,U8 U8,U9 U9,V1)
\tkzDrawSegments[col_a, line width=0.8pt, arrows=-Stealth](V1,V2 V2,V3 V3,V4 V4,V5 V5,V6 V6,V7 V7,V8 V8,V9 V9,U1)
\tkzDrawPoints[size=2.5pt, black](V1,V2,V3,V4,V5,V6,V7,V8,V9)
\tkzDrawPoints[size=2.5pt, black](U1,U2,U3,U4,U5,U6,U7,U8,U9)
\tkzDrawPoints[size=2.5pt, black](O,U,B,B')
\tkzDrawSegments[black, line width=0.5pt, densely dotted](V2,X2 V3,X3 V4,X4 V5,X5 V6,X6 V7,X7 V8,X8)
\tkzLabelPoint[left](U){$u$}
\tkzLabelPoint[xshift=-6pt,yshift=14pt](B){$b$}
\tkzLabelPoint[xshift=-8pt,yshift=0pt](B'){$-b$}
\tkzLabelLine[pos=1.2](O,A1){$a_1$}
\tkzLabelLine[pos=1.2](O,A2){$a_2$}
\tkzLabelLine[pos=1.2](O,A3){$a_3$}
\tkzLabelLine[pos=1.2](O,A4){$a_4$}
\tkzLabelLine[pos=1.2](O,A5){$a_5$}
\tkzLabelLine[pos=1.2](O,A6){$a_6$}
\tkzLabelLine[pos=1.2](O,A7){$a_7$}
\tkzLabelLine[pos=1.2](O,A8){$a_8$}
\tkzLabelLine[pos=1.2](O,A9){$a_9$}
\tkzLabelLine[pos=1.25,right](O,B){$U$}
\end{tikzpicture} 
\begin{tikzpicture}[scale=1.92]
\definecolor{col_a}{RGB}{195, 66, 63};
\tkzDefPoints{0/-1.4/BT,0/1.4/BB}
\tkzDrawSegment[opacity=0](BT,BB)
\begin{scope}[rotate=360-276.6296136769996]
\tkzDefPoints{0/0/O, 1/0/R}
\tkzDefPointsBy[rotation=center O angle -64](R){V1}
\tkzDefPointsBy[rotation=center O angle -40](R){V2}
\tkzDefPointsBy[rotation=center O angle -30](R){V3}
\tkzDefPointsBy[rotation=center O angle -20](R){V4}
\tkzDefPointsBy[rotation=center O angle -10](R){V5}
\tkzDefPointsBy[rotation=center O angle 20](R){V6}
\tkzDefPointsBy[rotation=center O angle 45](R){V7}
\tkzDefPointsBy[rotation=center O angle 65](R){V8}
\tkzDefPointsBy[rotation=center O angle 85](R){V9}
\tkzDefPointsBy[rotation=center O angle 180-64](R){U1}
\tkzDefPointsBy[rotation=center O angle 180-40](R){U2}
\tkzDefPointsBy[rotation=center O angle 180-30](R){U3}
\tkzDefPointsBy[rotation=center O angle 180-20](R){U4}
\tkzDefPointsBy[rotation=center O angle 180-10](R){U5}
\tkzDefPointsBy[rotation=center O angle 180+20](R){U6}
\tkzDefPointsBy[rotation=center O angle 180+45](R){U7}
\tkzDefPointsBy[rotation=center O angle 180+65](R){U8}
\tkzDefPointsBy[rotation=center O angle 180+85](R){U9}
\tkzDefPoints{0.01541888/-0.13266092/U}
\tkzInterLL(V1,U9)(O,U)\tkzGetPoint{B}
\tkzInterLL(V9,U1)(O,U)\tkzGetPoint{B'}
\tkzDefLine[perpendicular=through O,K=1](O,U) \tkzGetPoint{D}
\tkzInterLL(O,D)(V4,V5) \tkzGetPoint{P1}
\tkzInterLL(O,D)(U5,U4) \tkzGetPoint{P2}
\tkzDrawCircle[line width=0.5pt, gray](O,R)
\tkzDrawLines[line width=0.5pt, gray, dashed, add=0.1 and 0.1](B,B')
\tkzDrawSegment[line width=0.5pt, col_a, dashed, add=0.05 and 0.05](P1,P2)
\tkzDrawSegments[black, line width=0.75pt](V1,V2 V2,V3 V3,V4 V4,V5 V5,V6 V6,V7 V7,V8 V8,V9 V9,U1 U1,U2 U2,U3 U3,U4 U4,U5 U5,U6 U6,U7 U7,U8 U8,U9 U9,V1)
\tkzDrawPoints[size=2.5pt, black](V1,V2,V3,V4,V5,V6,V7,V8,V9)
\tkzDrawPoints[size=2.5pt, black](U1,U2,U3,U4,U5,U6,U7,U8,U9)
\tkzDrawPoints[size=2.5pt, black](O,U,B,B')
\tkzLabelPoint[above](U){$u$}
\tkzLabelLine[pos=1.15](O,U6){$v_1$}
\tkzLabelLine[pos=1.15](O,U7){$v_2$}
\tkzLabelLine[pos=1.15](O,U8){$v_3$}
\tkzLabelLine[pos=1.15](O,U9){$v_4$}
\tkzLabelLine[pos=1.15](O,V1){$v_5$}
\tkzLabelLine[pos=1.2](O,V2){$v_6$}
\tkzLabelLine[pos=1.2](O,V3){$v_7$}
\tkzLabelLine[pos=1.2](O,V4){$v_8$}
\tkzLabelLine[pos=1.2](O,V5){$v_9$}
\end{scope}
\end{tikzpicture}
\caption{On the left, polygon $P$ with each term of the sum $u = v_1 - v_2 + v_3 - \dotsb + v_n$ highlighted.
The linear span $U$ of the vector $u$ intersects the boundary of $P$ on $b$ and $-b$.
In the middle, the (oblique) projections $\pi_U(\pm a_i)$ into $U$ are shown to pack in the segment $[-b,b]$.
On the right, rotate so $U$ aligns with $x$-axis, reflect and relabel vectors so they are on the right half of the plane.}
\label{fig:balancing}
\end{figure}
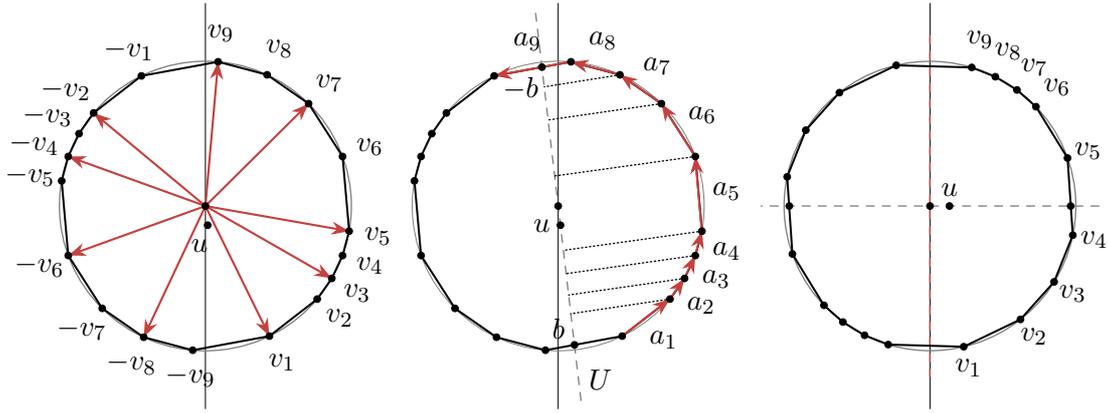

Putting this all together, we may assume that $u = \sum_{i=1}^{n}(-1)^{i-1} v_i = (\beta, 0)$ with $-1 \leq \beta \leq 1$ and $v_i = (\cos \theta_i, \sin \theta_i)$ where $-\pi/2 \leq \theta_1 \leq \dotsb \leq \theta_n \leq \pi/2$.
Write $v_i = (x_i, y_i)$ for every $1 \leq i \leq n$, so we have
\begin{equation}
\label{eq:beta}
    \beta = \sum_{i=1}^n (-1)^{i-1}x_i.
\end{equation}

We are going to analyse three cases: $\beta = -1$, $\beta = 1$ and $0 \leq \abs{\beta} < 1$.
For the first two cases, we will use the following observation: if there is a disjoint collection of pairs of indices $(i_1, j_1), \dotsc, (i_m, j_m)$ such that for every $1 \leq k \leq m$, we either have $v_{i_k} = v_{j_k}$ or $v_{i_k} = -v_{j_k}$, then we have $\prob[\big]{\norm{\sigma_V}_2 \leq 1} \geq 2^{m-n}$.
This is indeed the case as $\norm{\sum_{i=1}^{n}\eps_i v_i}_2 = \norm{u}_2 \leq 1$ for any sign sequence $\eps \in \set{-1,1}^n$ such that $\eps_{q} = (-1)^{q-1}$ if $q \notin \set{i_1, \dotsc, i_m, j_1, \dotsc, j_m}$, and such that $\eps_{i_k} = - \eps_{j_k}$ when $v_{i_k} = v_{j_k}$ and that $\eps_{i_k} = \eps_{j_k}$ when $v_{i_k} = -v_{j_k}$.

Suppose that $\beta = -1$.
Let $1 \leq 2r \leq n$ be the even index that maximises $x_{2r}$.
This implies that $x_{2k+1} \geq x_{2k}$ for every $1 \leq k < r$ and $x_{2k} \geq x_{2k-1}$ for every $r < k \leq (n-1)/2$.
In view of \eqref{eq:beta}, we have
\begin{equation}
\label{eq:beta=-1}
    -1 = \beta = x_1 + x_n + \sum_{k=1}^{r-1}(x_{2k+1} - x_{2k}) + \sum_{\mathclap{k=r+1}}^{\mathclap{(n-1)/2}}(x_{2k-1} - x_{2k}) - x_{2r} \geq -x_{2r} \geq -1,
\end{equation}
where equality occurs if and only if $v_1 = -v_n$, $v_{2k} = v_{2k+1}$ for $k \in [1,(n-1)/2]\setminus \set{r}$, and $v_{2r} = (1,0)$; see \Cref{fig:extreme} (left) for a concrete example.
This leads to $(n-1)/2$ disjoint pairs of identical or opposite vectors, which gives $\prob[\big]{\norm{\sigma_V}_2 \leq 1} \geq 2^{-(n-1)/2}$.

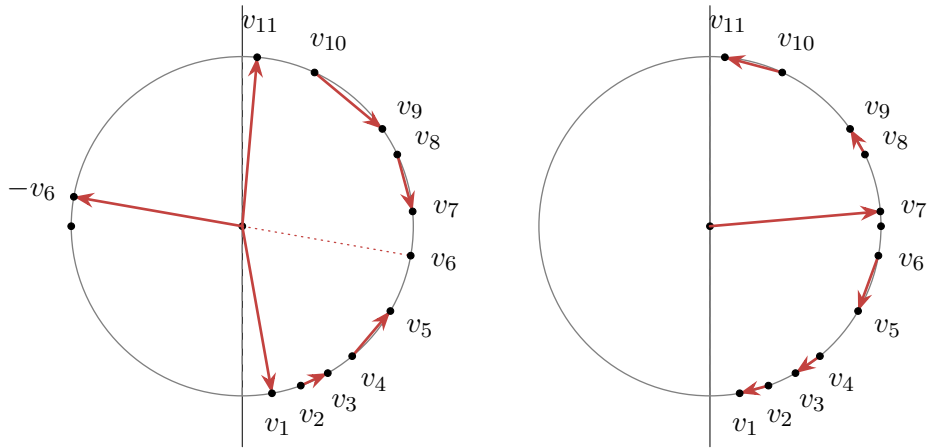
\begin{figure}[!ht]
\centering
\begin{tikzpicture}[scale=2.25]
\definecolor{col_a}{RGB}{195, 66, 63};
\tkzDefPoints{0/-1.3/BT,0/1.3/BB}
\tkzDrawSegment[opacity=0](BT,BB)
\tkzDefPoints{0/0/O, 1/0/R, -1/0/B}
\tkzDefPoints{0/-1/D, 0/1/U}
\tkzDefPointsBy[rotation=center O angle -80](R){V1}
\tkzDefPointsBy[rotation=center O angle -70](R){V2}
\tkzDefPointsBy[rotation=center O angle -60](R){V3}
\tkzDefPointsBy[rotation=center O angle -50](R){V4}
\tkzDefPointsBy[rotation=center O angle -30](R){V5}
\tkzDefPointsBy[rotation=center O angle -10](R){V6}
\tkzDefPointsBy[rotation=center O angle 5](R){V7}
\tkzDefPointsBy[rotation=center O angle 25](R){V8}
\tkzDefPointsBy[rotation=center O angle 35](R){V9}
\tkzDefPointsBy[rotation=center O angle 65](R){V10}
\tkzDefPointsBy[rotation=center O angle 85](R){V11}
\tkzDefPointsBy[rotation=center O angle 180-10](R){U6}
\tkzDrawCircle[line width=0.5pt, gray](O,R)
\tkzDrawSegment[line width=0.5pt, gray, dashed](D,U)
\tkzDrawPoints[size=2.5, fill=black](O,B)
\tkzDrawSegments[col_a, line width=1pt, arrows=-Stealth](O,V1 O,V11 V2,V3 V4,V5 O,U6 V8,V7 V10,V9)
\tkzDrawSegments[col_a, line width=0.5pt, dotted](O,V6)
\tkzDrawPoints[size=2.5, fill=black](V1,V2,V3,V4,V5,V6,V7,V8,V9,V10,V11,U6)
\tkzLabelLine[pos=1.2](O,V1){$v_1$}
\tkzLabelLine[pos=1.2](O,V2){$v_2$}
\tkzLabelLine[pos=1.2](O,V3){$v_3$}
\tkzLabelLine[pos=1.2](O,V4){$v_4$}
\tkzLabelLine[pos=1.2](O,V5){$v_5$}
\tkzLabelLine[pos=1.2](O,V6){$v_6$}
\tkzLabelLine[pos=1.2](O,V7){$v_7$}
\tkzLabelLine[pos=1.2](O,V8){$v_8$}
\tkzLabelLine[pos=1.2](O,V9){$v_9$}
\tkzLabelLine[pos=1.2](O,V10){$v_{10}$}
\tkzLabelLine[pos=1.2](O,V11){$v_{11}$}
\tkzLabelLine[pos=1.25](O,U6){$-v_6$}
\end{tikzpicture}
\qquad 
\begin{tikzpicture}[scale=2.25]
\definecolor{col_a}{RGB}{195, 66, 63};
\tkzDefPoints{0/-1.3/BT,0/1.3/BB}
\tkzDrawSegment[opacity=0](BT,BB)
\tkzDefPoints{0/0/O, 1/0/R, 1/0/B}
\tkzDefPoints{0/-1/D, 0/1/U}
\tkzDefPointsBy[rotation=center O angle -80](R){V1}
\tkzDefPointsBy[rotation=center O angle -70](R){V2}
\tkzDefPointsBy[rotation=center O angle -60](R){V3}
\tkzDefPointsBy[rotation=center O angle -50](R){V4}
\tkzDefPointsBy[rotation=center O angle -30](R){V5}
\tkzDefPointsBy[rotation=center O angle -10](R){V6}
\tkzDefPointsBy[rotation=center O angle 5](R){V7}
\tkzDefPointsBy[rotation=center O angle 25](R){V8}
\tkzDefPointsBy[rotation=center O angle 35](R){V9}
\tkzDefPointsBy[rotation=center O angle 65](R){V10}
\tkzDefPointsBy[rotation=center O angle 85](R){V11}
\tkzDrawCircle[line width=0.5pt, gray](O,R)
\tkzDrawSegment[line width=0.5pt, gray, dashed](D,U)
\tkzDrawPoints[size=2.5, fill=black](O,B)
\tkzDrawSegments[col_a, line width=1pt, arrows=-Stealth](V2,V1 V4,V3 V6,V5 V10,V11 V8,V9 O,V7)
\tkzDrawPoints[size=2.5, fill=black](V1,V2,V3,V4,V5,V6,V7,V8,V9,V10,V11)
\tkzLabelLine[pos=1.2](O,V1){$v_1$}
\tkzLabelLine[pos=1.2](O,V2){$v_2$}
\tkzLabelLine[pos=1.2](O,V3){$v_3$}
\tkzLabelLine[pos=1.2](O,V4){$v_4$}
\tkzLabelLine[pos=1.2](O,V5){$v_5$}
\tkzLabelLine[pos=1.2](O,V6){$v_6$}
\tkzLabelLine[pos=1.2](O,V7){$v_7$}
\tkzLabelLine[pos=1.2](O,V8){$v_8$}
\tkzLabelLine[pos=1.2](O,V9){$v_9$}
\tkzLabelLine[pos=1.2](O,V10){$v_{10}$}
\tkzLabelLine[pos=1.2](O,V11){$v_{11}$}
\end{tikzpicture}
\caption{The pairings involved in the characterisations of the case $\beta = -1$ (left) and the case $\beta = 1$ (right).}
\label{fig:extreme}
\end{figure}

Now suppose that $\beta = 1$ and let $1 \leq 2r+1 \leq n$ be the odd index that maximises $x_{2r+1}$.
This implies that $x_{2k} \geq x_{2k-1}$ for every $1 \leq k \leq r$ and $x_{2k} \geq x_{2k+1}$ for every $r < k \leq (n-1)/2$.
Again, using \eqref{eq:beta}, we have
\begin{equation}
\label{eq:beta=1}
    1 = \beta = x_{2r+1} + \sum_{k=1}^{r}(x_{2k-1} - x_{2k}) + \sum_{\mathclap{k=r+1}}^{\mathclap{(n-1)/2}}(x_{2k+1} - x_{2k})  \leq x_{2r+1} \leq 1,
\end{equation}
where equality occurs if and only if $v_{2r+1} = (1,0)$, and $v_{2k-1} = v_{2k}$ for $k \in  [1,r]$ and $v_{2k+1}=v_{2k}$ for $k \in [r+1,(n-1)/2]$; see \Cref{fig:extreme} (right) for a concrete example.
This leads to $(n-1)/2$ disjoint consecutive pairs of identical vectors, which gives $\prob[\big]{\norm{\sigma_V}_2 \leq 1} \geq 2^{-(n-1)/2}$.

From now onwards, we may assume that $0 \leq \abs{\beta} < 1$.
We define the following norm on $\RR^2$:
\begin{align*}
    \normstar{(x,y)} \defined \norm[\bigg]{ \p[\bigg]{\frac{x}{\sqrt{1-\abs{\beta}}}, \, y } }_2.
\end{align*}
In other words, $\normstar{}$ is the standard Euclidean norm $\norm{}_2$ after a carefully chosen stretching in the $x$ direction.
A simple fact is the following.
\begin{claim}
\label{cl:ellipse}
The $\normstar{}$-distance of the point $u$ to the $\norm{}$-unit circle is at least $\sqrt{1 - \abs{\beta}}$.
\end{claim}
\begin{proof}
Recall that $u = (\beta,0)$ and let $s \in \set{-1,1}$ be such that $\beta = s \abs{\beta}$.
For every $x, y \in \RR$ satisfying $x^2 + y^2 = 1$, we have
\begin{align*}
    \normstar{(x,y) - u} &= \sqrt{ \frac{(x - \beta)^2}{1 - \abs{\beta}} + y^2}
    = \sqrt{\frac{(x - \beta)^2 + (1 - \abs{\beta})(1 - x^2)}{1 - \abs{\beta}}} \\
    &= \sqrt{\frac{(1 - \abs{\beta})^2 + \abs{\beta}(1 - 2 s x + x^2)}{1-\abs{\beta}}} \\
    &= \sqrt{\frac{(1 - \abs{\beta})^2 + \abs{\beta}(1 - s x)^2}{1-\abs{\beta}}}
    \geq \sqrt{1 - \abs{\beta}},
\end{align*}
which finishes the proof.
\end{proof}

This means that if a vector $v \in \RR^2$ is such that $\normstar{v - u} \leq \sqrt{1 - \abs{\beta}}$, then $\norm{v}_2 \leq 1$.
Equivalently, if $\normstar{v} \leq \sqrt{1 - \abs{\beta}}$, then $\norm{u + v}_2 \leq 1$.
This fact is illustrated in \Cref{fig:ellipse}.

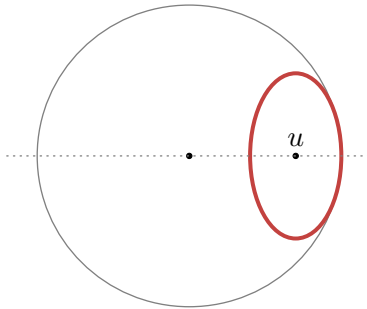
\begin{figure}[!ht]
\centering
\begin{tikzpicture}[scale=2]
\definecolor{col_a}{RGB}{195, 66, 63};
\def\bb{0.7}
\tkzDefPoints{0/0/O, 1/0/R, \bb/0/B}
\tkzDefPoints{-1.2/0/X1, 1.2/0/X2}
\tkzDefPoints{0/-1.1/Y1, 0/1.1/Y2}
\tkzDrawCircle[line width=0.5pt, gray](O,R)
\tkzDrawPoints[size=2, fill=black](O,B)
\tkzDrawEllipse[col_a, line width=1.5pt](B,1-\bb,{sqrt(1-\bb)},0)
\tkzDrawSegment[line width=0.5pt, gray, dotted](X1,X2)
\tkzLabelPoint[above](B){$u$}
\end{tikzpicture}
\caption{The $\normstar{}$-ball of radius $\sqrt{1 - \abs{\beta}}$ centred at $u = (\beta,0)$ is fully contained in the unit $\norm{}_2$-ball centred at the origin.}
\label{fig:ellipse}
\end{figure}

Call a collection $\cP$ of disjoint pairs of indices in $[n]$ a \emph{pairing}.
We say that a pairing $\cP$ is \emph{parity-balanced} if for every $(i,j) \in \cP$, $i$ and $j$ are of different parities, and \emph{suitable} if it is parity-balanced and
\begin{equation*}
    \sum_{(i,j) \in \cP} \normstar{v_i - v_j} \leq \frac{\sqrt{1 - \abs{\beta}}}{2}.
\end{equation*}
A key observation is the following.

\begin{claim}
\label{cl:SuitableMachtingIsEnough}
If $\cP$ is a suitable pairing, then $\prob[\big]{\norm{\sigma_V}_2 \leq 1} \geq 2^{\card{\cP}-n}$.
\end{claim}
\begin{proof}
Note that any sequence $\eta \from \cP \to \set{0,1}$ satisfies
\begin{equation}
\label{eq:signings}
    \norm[\Big]{u + 2 \! \sum_{(i,j) \in \cP} (-1)^{i} \, \eta(i,j) \, (v_{i} - v_{j}) }_2 \leq 1.
\end{equation}
Indeed, this follows from \Cref{cl:ellipse} since
\begin{equation*}
    \normstar[\Big]{\sum_{(i,j) \in \cP} (-1)^{i} \eta(i,j) (v_i - v_j)} \leq 2\sum_{(i,j) \in \cP} \normstar{v_{i} - v_{j}} \leq \sqrt{1 - \abs{\beta}}.
\end{equation*}
Finally, note that each expression in \eqref{eq:signings} is obtained from $u = \sum_{t = 1}^n (-1)^{t-1}v_t$ by flipping the signs of $v_i$ and $v_j$ when $\eta(i,j) = 1$.
This gives then at least $2^{\card{\cP}}$ signings whose sum lies in the unit disk, so we indeed have $\prob[\big]{\norm{\sigma_V}_2 \leq 1} \geq 2^{\card{\cP}-n}$.
\end{proof}

Before we find a suitable pairing, we show that we always have an almost maximal pairing $\cP$ that is almost suitable.

\begin{claim}
\label{cl:integral}
There is a parity-balanced pairing $\cP$ in $[n]$ with $\sum_{(i,j) \in \cP}\normstar{v_i - v_j} \leq \pi \sqrt{1 - \abs{\beta}}$ and $\card{\cP} \geq (n-3)/2$.
\end{claim}
\begin{proof}
Recall that $v_i = (\cos \theta_i,\sin \theta_i)$, and suppose that $v_i$ and $v_j$ are such that $0 \leq \theta_i \leq \theta_j$.
The $\normstar{}$-distance from $v_i$ to $v_j$ is at most the $\normstar{}$-length of the circular arc that connects $v_i$ to $v_j$, therefore
\begin{equation*}
    \normstar{v_i - v_j} \leq \int_{\theta_i}^{\theta_j} \sqrt{\normstar[\big]{\p[\big]{(\sin\theta)^2,(\cos\theta)^2}}} \dtheta = \int_{\theta_i}^{\theta_j} \sqrt{\frac{(\sin \theta)^2}{1 - \abs{\beta}} + (\cos \theta)^2} \dtheta.
\end{equation*}
By the change of variables $\theta = \arccos t$, the integral above is equal to
\begin{equation*}
    \int_{\theta_i}^{\theta_j} \sqrt{\frac{1 - \abs{\beta}(\cos\theta)^2}{1 - \abs{\beta}}} \dtheta
    = \int_{x_j}^{x_i} \sqrt{\frac{1 - \abs{\beta}t^2}{(1 - \abs{\beta})(1-t^2)}} \dt.
\end{equation*}
For convenience, we define $g \from [0,1) \to \RR_{\geq 0}$ as
\begin{equation*}
    g(t) \defined \sqrt{\frac{1 - \abs{\beta}t^2}{(1 - \abs{\beta})(1-t^2)}}
\end{equation*}
and note that $g$ is strictly increasing in its domain.
When $0 \leq \theta_i \leq \theta_j$, we have
\begin{equation}
\label{eq:length}
    \normstar{v_i - v_j} \leq \int_{x_j}^{x_i} g(t) \dt.
\end{equation}
Moreover, we note that if $\theta_i \leq \theta_j \leq 0$, instead of \eqref{eq:length}, we have
\begin{equation}
\label{eq:length2}
    \normstar{v_i - v_j} \leq \int_{x_i}^{x_j} g(t) \dt,
\end{equation}
while if $\theta_i \leq 0 \leq \theta_j$, instead of \eqref{eq:length}, we have
\begin{equation}
\label{eq:length3}
    \normstar{v_i - v_j} \leq  \normstar{v_i - (1,0)} + \normstar{(1,0) - v_j } \leq  \int_{x_i}^{1} g(t) \dt + \int_{x_j}^{1} g(t) \dt.
\end{equation}

Suppose that $0 \leq \beta < 1$, let $1 \leq 2r+1 \leq n$ be the odd index that maximises $x_{2r+1}$ and consider the pairing
\begin{equation*}
    \cP \defined \set[\big]{(v_{2k}, v_{2k-1}) \st 1 \leq k \leq r} \cup \set[\big]{(v_{2k},v_{2k+1}) \st r+1\leq k \leq (n-1)/2}.
\end{equation*}
Note that $\cP$ is parity-balanced and $\card{\cP} = (n-1)/2$.
It remains to show that the sum
\begin{equation*}
    S \defined \sum_{(i,j) \in \cP} \normstar{v_i - v_j} = \sum_{k=1}^{r}\normstar{v_{2k} - v_{2k-1}} + \sum_{\mathclap{k=r+1}}^{\mathclap{(n-1)/2}}\normstar{v_{2k} - v_{2k+1}}
\end{equation*}
is not too large.
Assume initially that for no pair $(i,j) \in \cP$ we have $\theta_i$ and $\theta_j$ of opposite signs.
From \eqref{eq:length} and \eqref{eq:length2}, we have
\begin{equation}
\label{eq:estimate}
     S \leq \sum_{k=1}^{r}\int_{x_{2k-1}}^{x_{2k}} g(t) \dt  + \sum_{\mathclap{k=r+1}}^{\mathclap{(n-1)/2}} \; \int_{x_{2k+1}}^{x_{2k}} g(t) \dt.
\end{equation}
As $g$ is increasing in $[0,1)$, each of the integrals in \eqref{eq:estimate} can be bounded from above by shifting the intervals of integration to the right as much as we can, while maintaining them internally disjoint.
In other words, if we define
\begin{equation*}
    I \defined \sum_{k=1}^{r}(x_{2k} - x_{2k-1}), \quad \text{ and } \quad I' \defined  \sum_{\mathclap{k=r+1}}^{\mathclap{(n-1)/2}}(x_{2k} - x_{2k+1}),
\end{equation*}
then from \eqref{eq:estimate} we have
\begin{equation}
\label{eq:estimate2}
    S \leq \int_{1-I}^{1} g(t) \dt  + \int_{1-I'}^{1} g(t) \dt,
\end{equation}
and moreover, we can use \eqref{eq:beta=1} to estimate
\begin{equation}
\label{eq:ineq-beta}
     I + I' = \sum_{k=1}^{r}(x_{2k} - x_{2k-1}) + \sum_{\mathclap{k=r+1}}^{\mathclap{(n-1)/2}}(x_{2k} - x_{2k+1}) = x_{2r+1} - \beta \leq 1 -\abs{\beta}.
\end{equation}
Since we can exchange mass from one integral to the other in \eqref{eq:estimate2}, we have
\begin{equation}
\label{eq:estimate3}
     S \leq 2\int_{1 - \frac{I + I'}{2}}^{1} g(t) \dt
     \leq 2\int_{\frac{1+\abs{\beta}}{2}}^{1} g(t) \dt.
\end{equation}
Before we proceed with the proof, note that the same estimate also holds if there is a pair $(v_i, v_j)$ with $\theta_i$ and $\theta_j$ of opposite signs.
Indeed, we would apply \eqref{eq:length3} instead of \eqref{eq:length} or \eqref{eq:length2} in the estimate \eqref{eq:estimate}.
In effect, this is equivalent to splitting the pair $(v_i,v_j)$ into $(v_i, (1,0))$ and $(v_j,(1,0))$, which does not affect the inequality \eqref{eq:ineq-beta} since we would be adding and subtracting $1$ from the left hand side.

Coming back to the integral, we obtain from \eqref{eq:estimate3} that
\begin{align*}
    S &\leq \frac{2}{\sqrt{1 - \abs{\beta}}} \int_{\frac{1+\abs{\beta}}{2}}^{1} \sqrt{\frac{1 - \abs{\beta}t^2}{1-t^2}} \dt
    \leq \frac{2\sqrt{1 - \abs{\beta} \p[\big] {\frac{1+\abs{\beta}}{2}}^2}}{\sqrt{1 - \abs{\beta}}} \int_{\frac{1+\abs{\beta}}{2}}^{1} \frac{1}{\sqrt{1-t^2}} \dt. \\
    &= \sqrt{\abs{\beta}^2 + 3\abs{\beta} + 4} \cdot \arccos \p[\big]{\frac{1+\abs{\beta}}{2}}
    \leq 2\sqrt{2} \arccos \p[\Big]{\frac{1+\abs{\beta}}{2}}.
\end{align*}
Since we have $\arccos(x) \leq (\pi/2) \sqrt{1 - x}$ for $0 \leq x \leq 1$, it follows that
\begin{align*}
    S  \leq \pi \sqrt{2} \sqrt{\frac{1}{2} - \frac{\abs{\beta}}{2}} = \pi \sqrt{1 - \abs{\beta}},
\end{align*}
as we claimed.

The case $-1 < \beta \leq 0$ is almost identical.
Let $1 \leq 2r \leq n$ be the even index that maximises $x_{2r}$ and consider the pairing
\begin{equation*}
    \cP' \defined \set[\big]{(v_{2k+1}, v_{2k}) \st 1 \leq k \leq r-1} \cup \set[\big]{(v_{2k-1},v_{2k}) \st r+1\leq k \leq (n-1)/2}.
\end{equation*}
We want to show that
\begin{equation*}
    S' \defined \sum_{(i,j) \in \cP'} \normstar{v_i - v_j} = \sum_{k=1}^{r-1}\normstar{v_{2k+1} - v_{2k}} + \sum_{\mathclap{k=r+1}}^{\mathclap{(n-1)/2}}\normstar{v_{2k-1} - v_{2k}}
\end{equation*}
is not very large.
Assume as before that for no pair $(v_i,v_j)$ just described, we have $\theta_i$ and $\theta_j$ of opposite signs.
From \eqref{eq:beta=-1}, we obtain the key estimate
\begin{equation*}
    \sum_{k=1}^{r-1}(x_{2k+1} - x_{2k}) + \sum_{\mathclap{k=r+1}}^{\mathclap{(n-1)/2}} (x_{2k-1} - x_{2k}) = \beta - x_1 - x_n + x_{2r} \leq 1 - \abs{\beta}.
\end{equation*}
From now onwards, we proceed identically to the case $\beta > 0$.
We use \eqref{eq:length}, \eqref{eq:length2}, and \eqref{eq:length3} to obtain
\begin{equation*}
    S' \leq \sum_{k=1}^{r-1}\int_{x_{2k}}^{x_{2k+1}} g(t) \dt  + \sum_{\mathclap{k=r+1}}^{\mathclap{(n-1)/2}} \; \int_{x_{2k}}^{x_{2k-1}} g(t) \dt
    \leq 2\int_{\frac{1+\abs{\beta}}{2}}^{1} g(t) \dt.
\end{equation*}
This gives then $S' \leq \pi \sqrt{1 - \abs{\beta}}$, but this time, we only have $\card{\cP'} = (n-3)/2$ rather than $\card{\cP} = (n-1)/2$.
\end{proof}

To complete the proof of \Cref{thm:lower-bound}, let $\cP$ be the pairing from \Cref{cl:integral} and partition $\cP$ into $7$ pieces, each of size at least $\floor{\card{\cP}/7}$.
At least one of these pieces $\cP'$ will be suitable, since $\pi/7 < 1/2$, which gives
\begin{equation*}
    \prob[\big]{\norm{\sigma_V}_2 \leq 1} \geq \frac{2^{\floor{\card{\cP}/7}}}{2^n} \geq \frac{1}{4} \cdot \frac{2^{n/14}}{2^n} = \frac{1}{4} \cdot \p[\big]{2^{-13/14}}^n,
\end{equation*}
and we are done as $2^{-13/14} \approx 0.5253$.
\end{proof}

\begin{remark}
\label{rmk:elliptic}
In the proof of \Cref{cl:integral}, we were a bit wasteful in estimating the integral $\int_\alpha^1 g(t) \dt$ for $\alpha = \arcsin(\frac{1+\abs{\beta}}{2})$.
In fact, this integral can be expressed explicitly as  $\p{E(\pi/2, \sqrt{\abs{\beta}}) - E(\alpha, \sqrt{\abs{\beta}})}/ \sqrt{1 - \abs{\beta}}$ where $E(x,k)$ is the incomplete elliptic integral of the second kind.
It seems numerically that the inequality $S \leq (\pi/3)\sqrt{1-\abs{\beta}}$ holds.
This improvement would propagate to a better exponent in \Cref{thm:lower-bound}.
Moreover, the last step of the argument where we partition $\cP$ into seven pieces has also a lot of slack.
A more delicate argument there also improves the constant in the base of the exponent.
\end{remark}


\section{Odd counterexamples to Erdős' conjecture}
\label{sec:radius1}

Our goal now is to prove \Cref{thm:radius1}, which demonstrates the double-jump phase transition as discussed in \Cref{sec:intro} and disproves \Cref{conj:HJNSodd}.
To do this, we in fact prove the more general \Cref{thm:construction-small-radius}, which we now restate for convenience.

\constsmallradius*
\begin{proof}
Let $k_1^+, k_1^-, k_2, \dotsc, k_d \geq 0$ be integers such that
\begin{equation*}
    k_1^+ + k_1^- + k_2 + \dotsb + k_d = n,
\end{equation*}
and such that $k_1^+$ and $k_2$ are even, while $k_1^-, k_3, \dotsb, k_d$ are all odd.
Denote by $e_1, \dotsc, e_d$ the standard basis vectors in $\RR^d$.
Let $0 < \beta < \pi/2$ be small enough so that $\sin \beta < 1/n$ and consider the perturbed basis vectors
\begin{align*}
    e_1^+ &\defined \p[\big]{ \cos\beta, \phantom{+}\sin\beta,\, 0,\, \dotsc,\, 0 }, \\
    e_1^- &\defined \p[\big]{ \cos\beta, -\sin\beta,\, 0,\, \dotsc,\, 0 }.
\end{align*}

The collection of vectors $v_1, \dotsc, v_n$ we consider consists of $k_1^+$ copies of $e_1^+$, $k_1^-$ copies of $e_1^-$, and $k_i$ copies of $e_i$ for $2 \leq i \leq d$.
It will be convenient to reparametrise the vectors in the following way.
Consider a partition of the indices
\begin{equation*}
    [n] = I_1^+ \sqcup I_1^- \sqcup I_2 \sqcup \dotsb \sqcup I_d,
\end{equation*}
where $\card{I_1^+} = k_1^+$, $\card{I_1^-} = k_1^-$, and $\card{I_i} = k_i$ for $2 \leq i \leq d$.
We write
\begin{align*}
    \sigma_1^+ \defined \sum_{i \in I_1^+} \eps_i e_1^+, \quad
    \sigma_1^- \defined \sum_{i \in I_1^-} \eps_i e_1^-, \quad
    \sigma_2 \defined \sum_{i \in I_2} \eps_i e_2, \quad
    \dotsb \quad
    \sigma_d \defined \sum_{i \in I_d} \eps_i e_d,
\end{align*}
and the sign sums
\begin{align*}
    \cE_1^+ \defined \sum_{i \in I_1^+} \eps_i, \quad
    \cE_1^- \defined \sum_{i \in I_1^-} \eps_i, \quad
    \cE_2 \defined \sum_{i \in I_2} \eps_i, \quad
    \dotsb \quad
    \cE_d \defined \sum_{i \in I_d} \eps_i.
\end{align*}
Let $\sigma \defined \sigma_1^+ + \sigma_1^- + \dotsb + \sigma_d$ and note that
\begin{equation*}
    \sigma = \p[\big]{(\cE_1^+ + \cE_1^-)\cos\beta , \, \cE_2 + (\cE_1^+ - \cE_1^-)\sin\beta, \, \cE_3, \, \dotsc, \, \cE_d}.
\end{equation*}
Finally, for $1 \leq i \leq d$, write $\pi_i \from \RR^d \to \RR$ for the $i$-coordinate projection.

The main goal now is to understand which constraints on the sequence $\eps_1, \dotsc, \eps_n$ are imposed by the condition
\begin{equation}
\label{eq:small-radius-condition}
\norm{\sigma}_2^2 = \pi_1(\sigma)^2 + \dotsb + \pi_d(\sigma)^2 \leq d - 1.
\end{equation}
First we deal with the coordinates with $i \geq 3$.
Notice that for such $i$, we have that $k_i$ is odd, so $\pi_i(\sigma) = \cE_i \in 2\ZZ + 1$.
Therefore, we have $\pi_i(\sigma)^2 \geq 1$ and if $\pi_i(\sigma)^2 \neq 1$, then $\pi_i(\sigma)^2 \geq 9$.
From \eqref{eq:small-radius-condition}, we then have
\begin{equation*}
    \pi_3(\sigma)^2 + \dotsb + \pi_d(\sigma)^2 \leq d - 1,
\end{equation*}
which implies that $\pi_i(\sigma)^2 = 1$ for all $i \geq 3$.
This imposes that
\begin{equation}
\label{eq:Value-Eps3-to-d}
\cE_3 = \pm 1, \; \cE_4 = \pm 1, \; \dotsc \, , \; \cE_{d-1} = \pm 1, \; \cE_d = \pm 1.
\end{equation}
Therefore, conditionally on \eqref{eq:Value-Eps3-to-d} being satisfied, we have that \eqref{eq:small-radius-condition} holds if and only if
\begin{equation}
\label{eq:smaller-radius-condition}
\pi_1(\sigma)^2 + \pi_2(\sigma)^2 \leq 1.
\end{equation}

Now we consider the first two coordinates of $\sigma$ in light of \eqref{eq:smaller-radius-condition}, starting with the second.
Recall that $\pi_2(\sigma) =  \cE_2 + (\cE_1^+ - \cE_1^-)\sin\beta$.
Since $k_2$ is even, we have $\cE_2 \in 2\ZZ$.
As $\sin \beta \leq 1/n$, we have $(\cE_1^+ - \cE_1^-)\sin\beta \leq (k_1^+ + k_1^-)/n < 1$, which implies that either $\cE_2 = 0$ or $\pi_2(\sigma)^2 > 1$.
The second option would violate \eqref{eq:smaller-radius-condition}, thus we have $\cE_2 = 0$.

Finally, notice that $k_1^+$ and $k_1^-$ have different parities, so $(\cE_1^+ + \cE_1^-)^2 \geq 1$ and $(\cE_1^+ - \cE_1^-)^2 \geq 1$.
But since $\cE_2 = 0$, condition \eqref{eq:smaller-radius-condition} translates to
\begin{equation*}
    \pi_1(\sigma)^2 + \pi_2(\sigma)^2  = (\cos\beta)^2(\cE_1^+ + \cE_1^-)^2 + (\sin\beta)^2(\cE_1^+ - \cE_1^-)^2 \leq 1,
\end{equation*}
which can only be satisfied if $(\cE_1^+ + \cE_1^-)^2 = (\cE_1^+ - \cE_1^-)^2 = 1$, in which case \eqref{eq:smaller-radius-condition} is an equality.
Indeed, this implies that $\cE_1^+ + \cE_1^- = \pm 1$ and $\cE_1^+ - \cE_1^- = \pm 1$.
As $k_1^+$ is even and $k_1^-$ is odd, we then must have $\cE_1^+ = 0$ and $\cE_1^- = \pm 1$.
Therefore, we have
\begin{equation*}
    \prob[\big]{\norm{\sigma}_2 \leq \sqrt{d-1}}
    = \prob[\big]{\cE_1^+ = 0, \cE_1^- = \pm 1, \cE_2 = 0, \cE_3 = \pm 1, \dotsc, \cE_d = \pm 1},
\end{equation*}
and since $\cE_1^+, \cE_1^-, \cE_2, \dotsc, \cE_d$ are all independent, we have
\begin{equation}
\label{eq:construction-prob}
    \prob[\big]{\norm{\sigma}_2 \leq \sqrt{d-1}} = \frac{1}{2^n} \binom{k_1^+}{k_1^+/2}
    \binom{k_1^- + 1}{(k_1^- + 1)/2}
    \binom{k_2}{k_2/2}
    \prod_{i \geq 3}\binom{k_i + 1}{(k_i + 1)/2}.
\end{equation}
To obtain \eqref{eq:construction-prob}, we have used that, as $k_2$ is even, we have $\prob{\cE_2 = 0} = \binom{k_2}{k_2/2}/2^{k_2}$ and, as $k_3$ is odd, we have
\begin{equation*}
    \prob{\cE_3 = \pm 1} = \p[\bigg]{\binom{k_3}{(k_3-1)/2}+\binom{k_3}{(k_3+1)/2}}/2^{k_3} = \binom{k_3+1}{(k_3+1)/2}/2^{k_3},
\end{equation*}
and similarly for all $\cE_1^+, \cE_1^-, \cE_2, \dotsc, \cE_d$ according to parity.
We take $k_1^+, k_1^-, k_2, \dotsc, k_d$ to be as close as possible to $n/(d+1)$ while adhering to the parity constraints discussed above.
Using the bound $\binom{2m}{m} \leq 4^m / \sqrt{\pi m}$, we have
\begin{align*}
    \prob[\big]{\norm{\sigma}_2 \leq \sqrt{d-1}}
    &\leq \frac{2^{d-1}}{\sqrt{\pi^{d+1} k_1^+ (k_1^- + 1) k_2 (k_3 + 1) \dotsb (k_d+1)/2^{d+1}}} \leq \frac{C_d}{n^{(d+1)/2}},
\end{align*}
for some constant $C_d >0$ depending only on $d$.
\end{proof}


\section{Orthogonal, simplicial and mixed constructions}
\label{sec:optimal-constructions}

In this section, we prove a collection of results regarding the choice of vectors $V = \set{v_1, \dotsc, v_n} \subseteq \SS^{d-1}$ which may minimise $\prob[\big]{\norm{\sigma_V}_2 \leq \sqrt{d}}$.
While we are able to establish several key results, there is still much left to be understood; see \Cref{sec:conclusion} for a thorough discussion of the problems which remain open.

This section has a significant number of proofs and is divided into a few subsections.
We start with a brief outline.
First, in \Cref{subsec:2d}, we prove \Cref{thm:BetterConstant}, which shows that simplicial constructions outperform orthogonal constructions in two dimensions.

We then proceed to the case of higher dimensions, and give a proof of \Cref{thm:OrthoBetterThanSimplex}, which states that orthogonal constructions perform better than simplicial constructions when the dimension is high enough.
To show this, we study the asymptotic behaviour of optimal orthogonal constructions in \Cref{subsec:orthogonal}, and then prove a lower bound for simplicial constructions in \Cref{subsec:simplicial}.

Finally, in \Cref{subsec:mixed} we prove \Cref{thm:MixedIsBetter}, which shows that constructions of mixed type outperform orthogonal constructions for all $d\geq 3$.

\subsection{Two dimensions}
\label{subsec:2d}

In this subsection, we give a proof of \Cref{thm:BetterConstant}, which states that, in two dimensions, the optimal choice of vectors to minimise $\prob{\norm{\sigma_V}_2 \leq \sqrt{2}}$ is not of orthogonal type, as it is outperformed by a construction of simplicial type.

\begin{proof}[Proof of \Cref{thm:BetterConstant}]
By performing a rotation, we may assume that $u_1=(1,0)$, $u_2=(-1/2,\sqrt{3}/2)$ and $u_3=(-1/2,-\sqrt{3}/2)$.
We are then given that $\sigma_V = \sum_{i=1}^{3k} \eps_i v_i$, where $\eps_1, \dotsc, \eps_{3k}$ are independent Rademacher random variables, and where the vectors $v_i$ are defined as
\begin{equation*}
    v_i\defined\begin{cases}
        u_1=(1,0) &\text{if }1\leq i\leq k,\\
        u_2 = (-1/2,\sqrt{3}/2) &\text{if }k+1\leq i\leq 2k,\\
        u_3 = (-1/2,-\sqrt{3}/2) &\text{if }2k+1\leq i\leq 3k.
    \end{cases}
\end{equation*}

We start by showing that $\norm{\sigma_V}_2\leq\sqrt{2}$ only occurs when $\norm{\sigma_V}_2 = 0$.
Indeed, define $s_1 \defined (\eps_1+ \dotsb + \eps_{k})/2$, $s_2 \defined (\eps_{k+1} + \dotsb + \eps_{2k})/2$, and $s_3 \defined (\eps_{2k + 1} + \dotsb + \eps_{3k})/2$.

\begin{claim}
\label{cl:sigmaVSmallis0}
We have $\norm{\sigma_V}_2\leq\sqrt{2}$ if and only if $s_1 = s_2 = s_3$.
\end{claim}
\begin{proof}
Let $s_1\defined (\eps_1+ \dotsb + \eps_{k})/2$, $s_2 \defined (\eps_{k+1} + \dotsb + \eps_{2k})/2$, and $s_3 \defined (\eps_{2k + 1} + \dotsb + \eps_{3k})/2$, noting that $s_1$, $s_2$, and $s_3$ are either all integers or all half-integers.
We then have that $\sigma_V = (2s_1-s_2-s_3,\sqrt{3}(s_2-s_3))$; it may thus be easily computed that
\begin{align*}
    \norm{\sigma_V}_2^2 &= 4(s_1^2 + s_2^2 + s_3^2 - s_1 s_2 - s_1 s_3 - s_2 s_3) \\
    &= 2\p[\big]{(s_1 - s_2)^2 + (s_1 - s_3)^2 + (s_2 - s_3)^2}.
\end{align*}
Each difference $s_i - s_j$, for $1\leq i < j \leq 3$, is an integer and it cannot be the case that exactly two of these differences are zero.
It is therefore clear that $\norm{\sigma_V}_2^2$ is either $0$ or at least $4$.
However, $\norm{\sigma_V}_2 = 0$ implies $s_1 = s_2 = s_3$.
\end{proof}
Therefore, we have$\norm{\sigma_V}_2\leq\sqrt{2}$ if and only if
\begin{equation*}
    \eps_1+ \dotsb + \eps_{k} = \eps_{k+1} + \dotsb + \eps_{2k} = \eps_{2k + 1} + \dotsb + \eps_{3k}.
\end{equation*}
We may therefore see that
\begin{equation}
\label{eq:BetterConstantFranel}
    \prob[\big]{\norm{\sigma_{V}}_2 \leq \sqrt{2}} = 2^{-n} \sum_{i=0}^{k}\binom{k}{i}^3.
\end{equation}
Finally, applying \Cref{prop:franel} with $q=3$, we obtain
\begin{equation*}
    \prob[\big]{\norm{\sigma_{V}}_2 \leq \sqrt{2}} = 2^{-n}\p[\big]{1 + o(1)}\frac{2^{3k + 1}}{k\pi \sqrt{3}} = \p[\big]{1 + o(1)}\frac{2\sqrt{3}}{n\pi},
\end{equation*}
as wanted.
\end{proof}
\begin{remark}
\label{rem:BetterConstant}
The proof above can be adjusted so it holds for any value of $n$ as follows.
We will take $k_1$ copies of $(1,0)$, $k_2$ copies of $(-\sqrt{3}/2,1/2)$, and $k_3$ copies of $(-\sqrt{3}/2,-1/2)$, where $k_1 + k_2 + k_3 = n$ and, for all $j \in \set{1,2,3}$, we have $k_j \equiv n \mod 2$ and $\abs{k_j -  n/3} = O(1)$.
The proof proceeds exactly as above, except that now the sum in \eqref{eq:BetterConstantFranel} becomes
\begin{equation*}
    \prob[\big]{\norm{\sigma_{V}}_2 \leq \sqrt{2}} = 2^{-n} \sum_{i \in \ZZ} \prod_{j=1}^3 \binom{k_j}{(k_j+x)/2 + i},
\end{equation*}
where $x = 1$ if $n$ is odd and $x = 0$ if $n$ is even.
\Cref{prop:perturbed-franel-2} may now be applied to deduce the same approximation as in the above proof, as required.
\end{remark}

\subsection{Orthogonal type}
\label{subsec:orthogonal}
Recall that a vector $V$ is of orthogonal type if, up to a global rotation, we have that for every $i$ there is $j$ such that $v_i = e_j$, where $e_1, \dotsc, e_d$ denotes the standard orthogonal basis of $\RR^d$.
A vector of orthogonal type is then characterised by the \emph{multiplicity vector} $(m_1, \dotsc, m_d)$, where $V$ consists of $m_i$ copies of $e_i$.
The first goal of this section is to determine the optimal choice of multiplicity vector for $V$ of orthogonal type when minimising $\prob[\big]{\norm{\sigma_V}_2 \leq \sqrt{d}}$.

We say that a multiplicity vector $(m_1, \dotsc, m_d)$ has \emph{parity vector} $(h_1, \dotsc, h_d) \in \set{0,1}^d$ if $m_i \equiv h_i \mod{2}$ for $i$.
The proposition below shows that the parity vector plays a crucial role in determining the optimal choice of $V$.
For a parity vector $h \in \set{0,1}^d$, denote by $\cS_h$ the set
\begin{equation*}
    \cS_h \defined \set[\Big]{(x_1, \dotsc, x_d) \in \ZZ^d \st x_1^2 + \dotsb + x_d^2 \leq d \text{ and } x_i \equiv h_i \mod{2} \text{ for all $i \in [d]$}}.
\end{equation*}
As we will now see, intuitively we want to choose $h$ such that $\card{\cS_h}$ is as small as possible.

\begin{proposition}
\label{prop:orthogonal-generic}
Fix $h \in \set{0,1}^d$ and let $m_1, \dotsc, m_d, n$ be positive integers with $m_1 + \dotsb + m_d = n$ and $m_i \equiv h_i \mod{2}$ for all $i$.
Let $V =  \set{v_1, \dotsc, v_n}$ consist of $m_i$ copies of $e_i$ for all $i$, then:
\begin{enumerate}
    \item If $m_i \to \infty$ as $n \to \infty$ for all $i$, then we have
    \begin{equation*}
        \prob[\big]{\norm{\sigma_V}_2 \leq \sqrt{d}} = \p[\big]{1 + o(1)} \card{\cS_h} \p[\Big]{ \displaystyle \frac{2^d}{\pi^d m_1 \dotsb m_d}}^{1/2}.
    \end{equation*}
    \item If for some $i$, we have $m_i = o(\log n)$ as $n \to \infty$, then
    \begin{equation*}
        \liminf_{n \to \infty} \prob[\big]{\norm{\sigma_V}_2 \leq \sqrt{d}}n^{d/2} = \infty.
    \end{equation*}
\end{enumerate}
\end{proposition}
\begin{proof}
Let $\sigma_1, \dotsc, \sigma_d$ denote the sums of the signs associated with the vectors $v_i$ that correspond to $e_1, \dotsc, e_d$, respectively.
Therefore, we may write
\begin{equation*}
    \sigma_V = \eps_1 v_1 + \dotsb + \eps_n v_n = \sigma_1 e_1 + \dotsb + \sigma_d e_d,
\end{equation*}
where
\begin{equation*}
    \prob[\big]{\sigma_i = x_i} = \begin{cases}
    \frac{1}{2^{m_i}}\binom{m_i}{(m_i+x_i)/2} & \text{if $x_i \equiv m_i \mod{2}$} \\
    0 & \text{if $x_i \not\equiv m_i \mod{2}$},
\end{cases}
\end{equation*}
Note that $\norm{\sigma_V}_2 \leq \sqrt{d}$ if and only if
\begin{equation*}
    \sigma_1^2 + \dotsb + \sigma_d^2 \leq d.
\end{equation*}
Note also that this inequality holds if and only if $(\sigma_1, \dotsc, \sigma_d) \in \cS_h$, and therefore
\begin{equation*}
    \prob[\big]{\norm{\sigma_V}_2 \leq \sqrt{d}} = \prob[\big]{\sigma \in \cS_h} = \!\! \sum_{(x_1,\dotsc,x_d) \in \cS_h} \, \prod_{i=1}^d \prob[\big]{\sigma_i = x_i}.
\end{equation*}
It follows that
\begin{equation}
\label{eq:identity-ortho}
    \prob[\big]{\norm{\sigma_V}_2 \leq \sqrt{d}} = \frac{1}{2^n} \! \sum_{(x_1,\dotsc,x_d) \in \cS_h} \, \prod_{i=1}^d \binom{m_i}{(m_i + x_i)/2}.
\end{equation}

We now split into cases, as in the statement of \Cref{prop:orthogonal-generic}, first considering the situation wherein all $m_i\to\infty$, and then the case in which $m_i = o(\log n)$ for some index $i$.

\begin{description}[leftmargin=0em]
\item[Case (i)]
If $x_i = h_i \mod{2}$ and $x_i = O(1)$, then it follows from \Cref{fact:stirling} that
\begin{equation*}
    \prob[\big]{\sigma_i = x_i} = \p[\big]{1 + o(1)} \sqrt{\frac{2}{\pi m_i}}.
\end{equation*}
Since solutions $(x_1, \dotsc, x_d) \in \cS_h$ have all the coordinates $x_i$ bounded, we are indeed in the above situation, and thus
\begin{align*}
    \prob[\big]{\norm{\sigma_V}_2 \leq \sqrt{d}} &= \sum_{(x_1,\dotsc,x_d) \in \cS(h)} \; \prod_{i=1}^d \p[\big]{1 + o(1)}\sqrt{\frac{2}{\pi m_i}}  \\
    &= \p[\big]{1 + o_d(1)} \card{\cS_h} \p[\Big]{\frac{2^d}{\pi^d m_1 \dotsb m_d}}^{1/2},
\end{align*}
This completes Case (i).
\item[Case (ii)]
Reorder indices so that there is a maximal $1 \leq t < d$ such that $m_1, \dotsc, m_t = o(\log n)$ and $m_i \to \infty$ for all $i \geq t+1$.
Note that $h \in \cS_h$ and that $\prob[\big]{\sigma_i = h_i} \geq 1/2^{m_i}$, and hence
\begin{align*}
    \prob[\big]{\norm{\sigma_V}_2 \leq \sqrt{d}} &\geq \prod_{i=1}^d \frac{1}{2^{m_i}} \binom{m_i}{(m_i + h_i)/2} \\
    &\geq \prod_{i=1}^t \frac{1}{2^{m_i}} \cdot \prod_{i=t+1}^d \frac{1}{2^{m_i}} \binom{m_i}{(m_i + h_i)/2} \\
    &\geq \p[\big]{1 + o_d(1)} \frac{1}{2^{m_1 + \dotsb + m_t}} \p[\Big]{\frac{2(d-t)}{\pi}}^{(d - t)/2} \frac{1}{n^{(d-t)/2}}.
\end{align*}
where in the last line, we have applied \Cref{prop:lower-bound-prod-binomial}.
Thus, as $m_1+\dotsb+m_t = o(\log n)$, we find that $\prob[\big]{\norm{\sigma_V}_2 \leq \sqrt{d}} = \omega(n^{-d/2})$, as claimed. \qedhere
\end{description}
\end{proof}

From \Cref{prop:orthogonal-generic}, it is clear that an optimal construction of orthogonal type must have $m_i = \p[\big]{1/d + o(1)}n$ for all $i$; one may deal with cases wherein $m_i \not\to \infty$ and $m_i \neq o(\log n)$ by passing into a subsequence.
Therefore we now assume $m_i = n/d + o(n)$ and focus on determining for which parity vector $h \in \set{0,1}^d$ the set $\cS_h$ is as small as possible.
For a particular value of $n$, however, we can only choose $h \in \set{0,1}^d$ such that $h_1 + \dotsb + h_d$ has the same parity as $n$, and so we consider the cases of $n$ odd and even separately.
We consider the following quantities
\begin{equation*}
    f_0(d) \defined \;\; \min_{\mathclap{\substack{h \in \set{0,1}^d\\h_1 + \dotsb + h_d \equiv d (2)}}} \quad \card{\cS_h}
    \quad \text{and } \quad
    f_1(d) \defined \;\; \min_{\mathclap{\substack{h \in \set{0,1}^d\\h_1 + \dotsb + h_d \not\equiv d (2)}}} \quad \card{\cS_h}.
\end{equation*}
In other words, $f_0(d)$ is the minimum of $\card{\cS_h}$ when $h$ has an even number of zeros and $f_1(d)$ is the minimum of $\card{\cS_h}$ when $h$ has an odd number of zeros.
We will make use of the following inequality.

\begin{proposition}
\label{prop:optimal-orthogonal-ineq}
We have that $(15/2^7) 2^d \leq f_0(d),f_1(d) \leq 2^d$ for every $d \geq 1$.
\end{proposition}

We obtain \Cref{prop:optimal-orthogonal-ineq} by fully determining $f_0(d)$ and $f_1(d)$, see \Cref{prop:optimal-orthogonal}.
As the proof is long and not particularly enlightening, we defer it to \Cref{app:orthogonal}.

From \Cref{prop:orthogonal-generic,prop:optimal-orthogonal-ineq}, one easily derives the following two corollaries. The first one will be used to prove \Cref{thm:OrthoBetterThanSimplex}, and the second one will be used to prove \Cref{thm:MixedIsBetter}.

\begin{corollary}
\label{cor:best-orthogonal}
There exists an absolute constant $C>0$ such that, for any integers $d\geq 1$ and $n \geq 1$, there is a choice of vectors $V \subseteq (\SS^{d-1})^n$ of orthogonal type such that
\begin{equation*}
    \prob[\big]{\norm{\sigma_V}_2 \leq \sqrt{d}} \leq C 2^d \p[\bigg]{\frac{2d}{\pi n}}^{d/2}.
\end{equation*}
\end{corollary}
\begin{proof}
Consider a set of $n$ vectors $V$ that consists of $n/d + O(1)$ copies of each vector of the orthogonal basis with a parity vector $h$ that minimises $\card{\cS_h}$.
From \Cref{prop:orthogonal-generic}, we have
\begin{align*}
    \prob[\big]{\norm{\sigma_V}_2 \leq \sqrt{d}} = \p[\big]{1 + o(1)} \card{\cS_h} \p[\Big]{ \displaystyle \frac{2^d}{\pi^d (n/d)^d}}^{1/2}.
\end{align*}
Either $\card{\cS_h} = f_0(d)$ or $\card{\cS_h} = f_1(d)$, depending on the parity of $n$.
\Cref{prop:optimal-orthogonal-ineq} implies that
\begin{align*}
    \prob[\big]{\norm{\sigma_V}_2 \leq \sqrt{d}} \leq C 2^d \p[\bigg]{\frac{2d}{\pi n}}^{d/2},
\end{align*}
for some absolute constant $C>0$, as we wanted.
\end{proof}

\begin{corollary}
\label{cor:lower-bound-ortho}
For any $C' < 15/2^7$ and any $d\geq 1$, for $n$ sufficiently large and any family of vectors $V \subseteq (\SS^{d-1})^n$ of orthogonal type we have that
\begin{equation*}
    \prob[\big]{\norm{\sigma_V}_2 \leq \sqrt{d}} \geq C' 2^d \p[\bigg]{\frac{2d}{\pi n}}^{d/2}.
\end{equation*}
\end{corollary}
\begin{proof}
Our starting point is identity \eqref{eq:identity-ortho} from \Cref{prop:orthogonal-generic}, that states
\begin{align*}
    \prob[\big]{\norm{\sigma_V}_2 \leq \sqrt{d}} = \frac{1}{2^n} \sum_{(x_1,\dotsc,x_d) \in \cS_h} \prod_{1 \leq i \leq d} \binom{m_i}{(m_i + x_i)/2}.
\end{align*}
Note that for every $(x_1, \dotsc, x_d) \in \cS_h$, we have $\abs{x_i} \leq \sqrt{d}$ for every $i$.
Therefore, by \Cref{prop:lower-bound-prod-binomial} with $q = d$, we have
\begin{align*}
    \prob[\big]{\norm{\sigma_V}_2 \leq \sqrt{d}} \geq \sum_{(x_1,\dotsc,x_d) \in \cS_h} \p[\big]{1 + o(1)} \p[\bigg]{\frac{2d}{\pi n}}^{d/2} = \p[\big]{1 + o(1)} \card{\cS_h} \p[\bigg]{\frac{2d}{\pi n}}^{d/2}.
\end{align*}
Applying \Cref{prop:optimal-orthogonal-ineq}, we get
\begin{align*}
    \prob[\big]{\norm{\sigma_V}_2 \leq \sqrt{d}} \geq C' 2^d \p[\bigg]{\frac{2d}{\pi n}}^{d/2},
\end{align*}
for $n$ sufficiently large, and for any $C' < 15/2^7$.
\end{proof}

\begin{remark}
The expressions for $f_0(d)$ and $f_1(d)$ are given in \Cref{app:orthogonal}.
Perhaps surprisingly, we have that $f_0(d) > f_1(d)$ for all $d \neq 2$.
Also maybe unexpectedly, it follows from the proof of \Cref{prop:optimal-orthogonal} that optimal parity vectors $h$ that minimise $f_0(d)$ or $f_1(d)$ are unique, except at dimension $17$, where taking $h$ with $10$ zeros or $6$ zeros leads to the same bound for $f_0(17)$.
\end{remark}

\subsection{Simplicial type}
\label{subsec:simplicial}

We now investigate the asymptotic growth of $\prob[\big]{\norm{\sigma_V}_2 \leq \sqrt{d}}$ in the case where $V$ is of simplicial type.
In particular, our main result in this subsection is the following.

\begin{proposition}
\label{prop:simplicial-lower-bound}
There exists $d_0 \geq 0$ such that for any integer $d \geq d_0$ and $n$ sufficiently large in terms of $d$, if $V \subseteq (\SS^{d-1})^n$ is of simplicial type, then
\begin{equation}
\label{eq:simplicial-lb}
    \prob[\big]{\norm{\sigma_V}_2 \leq \sqrt{d}} \geq 2^{1.01d}\p[\bigg]{\frac{2}{\pi n}}^{d/2} (d+1)^{(d-1)/2}.
\end{equation}
\end{proposition}
\begin{proof}
Recall that a collection $V$ of vectors $v_1,\dotsc,v_n$ is of simplicial type if for every $i$ there is $j$ such that $v_i = u_j$, where $u_1, \dotsc, u_{d+1}$ are the vertices of a $d$-simplex in $\RR^d$ inscribed in the unit sphere.
The key property of these vectors that we use is that if $1 \leq i \neq j \leq d+1$ then $\inner{u_i}{u_j} = -1/d$.
Let $(m_1, \dotsc, m_{d+1})$ be the multiplicities of the vectors $u_i$ in $V$, and let $(h_1, \dotsc, h_{d+1}) \in \set{0,1}^{d+1}$ be the parity vector of the $m_i$.

Let $\sigma_1, \dotsc, \sigma_{d+1}$ denote the sum of the signs corresponding to $u_1, \dotsc, u_{d+1}$ respectively.
Thus $\sigma = \sigma_1 u_1 + \dotsb + \sigma_{d+1} u_{d+1}$, where the $\sigma_i$ are independent.
Note that
\begin{align}
\label{eq:simplicial-norm-expansion}
    \norm{\sigma}_2^2 = \inner{\sigma}{\sigma} = \sum_{i,j \in [d+1]} \sigma_i \sigma_j \inner{u_i}{u_j}  &= \sum_{i \in [d+1]}\sigma_i^2 - \frac{2}{d}\sum_{1 \leq i \neq j \leq d+1} \sigma_i \sigma_j \nonumber \\
    &= \frac{1}{d} \sum_{1 \leq i < j \leq d+1} (\sigma_i - \sigma_j)^2.
\end{align}
Thus $\norm{\sigma}_2^2 \leq d$ if and only if $\sum_{1 \leq i < j \leq d+1} (\sigma_i - \sigma_j)^2 \leq d^2$.
We therefore define the following set of solutions to the resulting quadratic inequality.
\begin{equation*}
    \cQ_h \defined \set[\Big]{(x_1,\dotsb,x_{d+1}) \in \ZZ^{d+1} \st \sum_{\mathclap{1 \leq i < j \leq d+1}} (x_i - x_j)^2 \leq d^2 \text{ and } x_i \equiv h_i \mod{2} \text{ for all $i$}}.
\end{equation*}
Consider the equivalence relation $\sim$ on $\ZZ^{d+1}$, defined by setting $(x_1, \dotsc, x_{d+1}) \sim (y_1, \dotsc, y_{d+1})$ if and only if $x_i - y_i = x_j - y_j$ for all $1 \leq i,j \leq d+1$.
Each element $(x_1, \dotsc, x_{d+1}) \in \cQ_h$ has one representative on $\cQ_h/\sim$ with $x_{d+1} = h_{d+1}$.
We may then define the set of such representatives
\begin{equation*}
    \cQ_h^\ast = \set[\big]{(x_1,\dotsc,x_{d+1}) \in \cQ_h \st x_{d+1} = h_{d+1}}.
\end{equation*}

From the above we obtain that
\begin{align}
\label{eq:simplex-prob-1}
    \prob[\big]{\norm{\sigma_V}_2 \leq \sqrt{d}}
    &= \sum_{(x_1,\dotsc,x_{d+1}) \in \cQ_h} \prob[\big]{\sigma_1 = x_1, \dotsc, \sigma_{d+1} = x_{d+1}} \nonumber \\
    &= \sum_{(x_1,\dotsc,x_{d+1}) \in \cQ_h^\ast} \sum_{k \in \ZZ} \, \prob[\big]{\sigma_1 = x_1 + k, \dotsc, \sigma_{d+1} = x_{d+1} + k} \nonumber \\
    &= \sum_{(x_1,\dotsc,x_{d+1}) \in \cQ_h^\ast} \sum_{k \in \ZZ} \, \prod_{i = 1}^{d+1}\frac{1}{2^{m_i}} \binom{m_i}{(m_i + x_i)/2 + k}.
\end{align}
We will now apply \Cref{prop:lower-bound-perturbed-franel} (with $q = d + 1$) to provide a lower bound on \eqref{eq:simplex-prob-1}.
To do so, we need to further restrict $\cQ_h^\ast$ to the set
\begin{equation*}
    \cQ_h^{\ast\ast} = \set[\big]{(x_1,\dotsc,x_{d+1}) \in \cQ_h^\ast \st \abs{x_i} \leq m_i \text{ for all $i$}}.
\end{equation*}
Therefore, we get
\begin{align}
\label{eq:lower-bound-with-Qh}
    \prob[\big]{\norm{\sigma_V}_2 \leq \sqrt{d}}
    &\geq \sum_{(x_1,\dotsc,x_{d+1}) \in \cQ_h^{\ast\ast}} \sum_{k \in \ZZ} \, \prod_{i = 1}^{d+1}\frac{1}{2^{m_i}} \binom{m_i}{(m_i + x_i)/2 + k} \nonumber\\
    &\geq \card{\cQ_h^{\ast \ast}} \, \frac{1}{2} \p[\bigg]{\frac{2}{\pi n}}^{d/2} (d+1)^{(d-1)/2} g(n),
\end{align}
where $g(n) = 1$ if $m_i \leq 2$ for at most one index $i$, and $g(n) = n^{1/8}$ otherwise.
If $g(n) = n^{1/8}$, then we are done, as $n$ is large relative to $d$, and, in particular, $n^{1/8} > 2^{1.01d}$.
Therefore we may assume in the rest of the proof that $g(n) = 1$, and that we have $m_i \leq 2$ for at most one index $i$.
Thus, to finish the proof, it suffices to show that $\card{\cQ_h^{\ast\ast}} \geq 2^{1.01d+1}$.

From  the parity vector $(h_1,\dotsc,h_{d+1}) \in \set{0,1}^{d+1}$, define numbers $x,y \in [0,1]$ such that
\begin{align*}
    \abs[\big]{\set{i \in [d] \st h_i \equiv 0 \mod 2}} &= xd,\text{ and}\\
    \abs[\big]{\set{i \in [d] \st h_i \equiv 1 \mod 2}} &= yd.
\end{align*}
Note in particular that $x+y=1$, and that we are not considering $h_{d+1}$, as $x_{d+1}$ is fixed in the definition of $\cQ_h^\ast$.
Define, for $j \in \ZZ$, the number
\begin{equation*}
    n_j \defined \abs[\big]{\set{i \in [d] \st x_i = j}}.
\end{equation*}

We define $\cA$ to be the set of vectors $(x_1, \dotsc, x_d, x_{d+1})$ such that $x_{d+1} = h_{d+1}$ and
\begin{equation}
\label{eq:atom-direction-counts}
    (n_{-2},n_{-1},n_0,n_1,n_2,n_3) = \p[\big]{0.02xd, 0.3yd, 0.68xd, 0.68yd, 0.3xd, 0.02yd}.
\end{equation}
Note in particular that this implies that, for all $i \in [d+1]$, $-2 \leq x_i \leq 3$.
We will moreover assume that each $n_i$ is an integer and at least 1 for $-2 \leq i \leq 3$.
The error terms resulting from the required rounding are insignificant, and so to maintain clarity of presentation we will make no further comment on them.
We prove two claims about the set $\cA$.

\begin{claim}
\label{cl:atoms-in-q}
We have $\cA \subseteq \cQ_h^\ast$.
\end{claim}
\begin{proof}
It suffices to prove that any point $(x_1, \dotsc, x_{d+1})$ satisfying \eqref{eq:atom-direction-counts} is in $\cQ_h^\ast$.
Assume first that $x$ and $y$ are both non-zero.
First, we have by definition of $\cA$ that $x_i \equiv h_i \mod{2}$ for all $i$.
Therefore, it suffices to show that
\begin{equation}
\label{eq:weight}
    \sum_{-2\leq i < j\leq 3} (j - i)^2 n_i n_j \leq d^2.
\end{equation}

We now define constants $I$ and $C$, corresponding to `internal' and `cross' terms in the sum in \eqref{eq:weight}, as follows.

\begin{equation*}
    I = \hspace{-1.3em} \sum_{\substack{-2\leq i < j\leq 3 \\ i\equiv j \equiv 0 \mod 2}} \hspace{-1.3em} \frac{(j - i)^2 n_i n_j}{x^2 d^2}
    = \hspace{-1.3em} \sum_{\substack{-2\leq i < j\leq 3 \\ i\equiv j \equiv 1 \mod 2}} \hspace{-1.3em} \frac{(j - i)^2 n_i n_j}{y^2 d^2}
    \quad \text{and}\quad
    C = \hspace{-1em} \sum_{\substack{-2\leq i < j\leq 3 \\ i\not\equiv j \mod 2}} \hspace{-1em} \frac{(j - i)^2 n_i n_j}{x y d^2}.
\end{equation*}
An elementary computation shows that $I=0.9664$ and $C= 1.9472 = 2I + \eps$ for some $\eps \in (0,0.015)$.
Therefore, we have
\begin{align}
\label{eq:atom-distance-expansion}
    \sum_{-2\leq i < j\leq 3} (j - i)^2 n_i n_j &= \p[\big]{I(x^2+y^2) + Cxy}d^2 \nonumber \\
    &= \p[\big]{I(x+y)^2 +\eps x y}d^2.
\end{align}
As $x + y = 1$, note that $(I(x+y)^2 +\eps x y)d^2$ is maximised when $x = y = 0.5$.
Using that $\eps < 0.015$, we obtain
\begin{align*}
    \sum_{-2\leq i < j\leq 3} (j - i)^2 n_i n_j \leq 0.99d^2 < d^2,
\end{align*}
as desired.
Note that, in the case where $x = 0$ or $y = 0$, the equality \eqref{eq:atom-distance-expansion} still holds, and the conclusion still follows.
\end{proof}

\begin{claim}
\label{cl:lots-of-atoms}
If $m_i\geq 3$ for all but at most one $m_i$, then for sufficiently large $d$, we have $\card{\cA \cap \cQ_h^{\ast\ast}} > 2^{1.011d}$.
\end{claim}
\begin{proof}
Given that $\cA \subseteq \cQ_h^\ast$, we have that $\cA \cap \cQ_h^{\ast\ast}$ consists of vectors $(x_1,
\dotsc, x_{d+1})$ such that $\abs{x_i} \leq m_i$ for all $i$.
It follows from the definition of $\cA$ in \eqref{eq:atom-direction-counts} that
\begin{align*}
    \card{\cA \cap \cQ_h^{\ast\ast}} \geq \binom{xd}{0.02xd,0.68xd,0.3xd}\binom{yd}{0.02yd,0.68yd,0.3yd}d^{-1}.
\end{align*}
Indeed, note that it is possible that $m_i \leq 1$ for some (unique) $i$, in which case we may have no freedom in choosing $x_i$.
However, if, say $m_1\leq 1$, then we may ignore this fact at first, and then apply a cyclic permutation to our choice so that $\abs{x_1}\leq 1$; hence the factor of $d^{-1}$.
Applying Stirling's formula, as given in \Cref{fact:stirling}, we have for some absolute constant $C>0$ that
\begin{equation*}
    \binom{m}{0.02 m, 0.68 m, 0.3 m} \geq C \cdot{\frac{2^{m H_3 (\alpha, \beta, \gamma)}}{m}},
\end{equation*}
where $H_3$ is the ternary entropy function, defined as
\begin{equation*}
    H_3(p,q,r) \defined - p \log p - q \log q - r \log r.
\end{equation*}
The numerical inequality $H_3(0.02,0.68,0.3) = 1.01231\dotsc > 1.012$ can be easily verified, we have
\begin{align*}
     \card{\cA \cap \cQ_h^{\ast\ast}} \geq \frac{C 2^{xd H_3(0.02,0.68,0.3)}}{xd} \cdot \frac{C 2^{yd H_3(0.02,0.68,0.3)}}{yd} \cdot \frac{1}{d}
     \geq C^2 \cdot {\frac{2^{ 1.012 d }}{xyd^3}} \geq \frac{C^2}{4} \cdot {\frac{2^{ 1.012 d }}{d^3}},
\end{align*}
which is at least $2^{1.011d}$ when $d$ is sufficiently large, as required.
\end{proof}

\Cref{prop:simplicial-lower-bound} now follows immediately from combining \Cref{cl:atoms-in-q} and \Cref{cl:lots-of-atoms}.
\end{proof}

With \Cref{cor:best-orthogonal} and \Cref{prop:simplicial-lower-bound} in hand, we may now deduce \Cref{thm:OrthoBetterThanSimplex}.

\begin{proof}[Proof of \Cref{thm:OrthoBetterThanSimplex}]
Let $Y$ be the choice of vectors given by \Cref{cor:best-orthogonal}. Then we have
\begin{align}
\label{eq:IneqSigmaY-comparison}
    \prob[\big]{\norm{\sigma_Y}_2 \leq \sqrt{d}} \leq C 2^d \p[\bigg]{\frac{2d}{\pi n}}^{d/2}.
\end{align}
By \Cref{prop:simplicial-lower-bound}, we have
\begin{align}
\label{eq:IneqSigmaX-comparison}
    \prob[\big]{\norm{\sigma_X}_2 \leq \sqrt{d}} \geq 2^{1.01d}\p[\bigg]{\frac{2}{\pi n}}^{d/2} (d+1)^{(d-1)/2}.
\end{align}
Let $\eps_d = 2^{-0.005d}$, and note that for $d$ sufficiently large, we have
\begin{align*}
    \eps_d 2^{1.01d}\p[\bigg]{\frac{2}{\pi n}}^{d/2} (d+1)^{(d-1)/2} \geq C 2^d \p[\bigg]{\frac{2d}{\pi n}}^{d/2}.
\end{align*}
Therefore, \eqref{eq:IneqSigmaY-comparison} and \eqref{eq:IneqSigmaX-comparison} imply together that
\begin{align*}
    \prob[\big]{\norm{\sigma_Y}_2 \leq \sqrt{d}} \leq \eps_d \prob[\big]{\norm{\sigma_X}_2 \leq \sqrt{d}},
\end{align*}
as wanted.
\end{proof}

\subsection{Mixed type}
\label{subsec:mixed}

We now prove \Cref{thm:MixedIsBetter}, which we may recall states that orthogonal constructions are never optimal.
Indeed, for $d=2$ this follows from \Cref{thm:BetterConstant}.
For $d=3$, one may follow a method similar to that used to prove \Cref{thm:BetterConstant} to show that the simplicial construction outperforms the orthogonal construction in three dimensions as well.
However, we show in the next proof that a hybrid construction, i.e. combining simplicial and orthogonal components, performs better than a pure orthogonal construction in $d\geq 3$ dimensions.
We emphasise that no effort has been made to find an optimal construction and the purpose of this section is merely to demonstrate that more complex constructions can outperform both orthogonal and simplicial constructions.

\begin{proof}[Proof of \Cref{thm:MixedIsBetter}]
As $7/2^6 < 15/2^7$, we have by \Cref{cor:lower-bound-ortho} for $n$ sufficiently large that
\begin{align}
\label{eq:formula-orthogo-mixedproof}
    \prob[\big]{\norm{\sigma_Y}_2 \leq \sqrt{d}} \geq \frac{7}{2^6} 2^d \p[\bigg]{\frac{2d}{\pi n}}^{d/2}.
\end{align}
We now describe the set of vectors $Z$ we consider.
Let $e_1, \dotsc, e_d$ be an orthonormal basis of $\RR^d$.
Take $a_1$, $a_2$, and $a_3$ copies of each of $w_1$, $w_2$, and $w_3$ respectively, which are the elements of the regular 2-simplex centred at the origin in span of $e_1$ and $e_2$, and $b_i$ copies of $e_i$ for $3 \leq i \leq d$.
We choose $a_1,a_2,a_3,b_3, \dotsc, b_d$ such that $a_1 \equiv a_2 \equiv a_3 \mod 2$, $b_3 \equiv 0 \mod 2$, and $b_4 \equiv \dotsb \equiv b_d \equiv 1 \mod 2$ and moreover, $a_1$, $a_2$ and $a_3$ are $2n/3d + O(1)$, while all $b_i$ are $n/d + O(1)$.
Note that, due to the freedom in whether the $a_i$ are even or odd, this construction is valid d $n$ is even or odd.

Let $\RR^d = W \oplus T \oplus U$, where $W$ is spanned by $e_1$ and $e_2$, $T$ is spanned by $e_3$, and $U$ is spanned by $e_4, \dotsc, e_d$.
Write $\sigma_W$, $\sigma_T$ and $\sigma_U$ for the orthogonal projections of $\sigma_Z$ into subspaces $W$, $T$ and $U$ respectively.
Notice that $\norm{\sigma_{Z}}_2 \leq \sqrt{d}$ is equivalent to
\begin{equation*}
    \norm{\sigma_W}_2^2 + \norm{\sigma_T}_2^2 + \norm{\sigma_U}_2^2 \leq d,
\end{equation*}
Since $b_i$ is odd for all $4\leq i \leq d$, we have $\norm{\sigma_U}_2^2 \geq d - 3$.
Moreover, as $b_3$ is even, we have either $\norm{\sigma_T}_2^2 = 0$ or $\norm{\sigma_T}_2^2 \geq 4$.
Finally, we may consider the lattice generated by $w_1$, $w_2$, and $w_3$ to see that either $\norm{\sigma_W}_2^2 = 0$ or $\norm{\sigma_W}_2^2 \geq 4$, similarly as in the proof of \Cref{cl:sigmaVSmallis0}.

Putting the above points together, we find that $\norm{\sigma_{Z}}_2 \leq \sqrt{d}$ is equivalent to
\begin{equation*}
    \norm{\sigma_W}_2^2 = 0, \quad \norm{\sigma_T}_2^2 = 0, \quad\text{ and }\quad \norm{\sigma_U}_2^2 = d - 3.
\end{equation*}
This allows us to apply \Cref{thm:BetterConstant} and \Cref{prop:orthogonal-generic} to find the following.
\begin{align}
\label{eq:formula-Mixed}
    \prob[\big]{\norm{\sigma_{Z}}_2 \leq \sqrt{d}} &= \p[\big]{1 + o(1)} \p[\Big]{\frac{2\sqrt{3}}{(2n/d)\pi}} 2^{d-3}\p[\Big]{\frac{2}{\pi(n/d)}}^{(d-2)/2} \nonumber \\
    &= \p[\big]{1 + o(1)}\frac{\sqrt{3}}{16}\p[\Big]{\frac{2 d}{\pi n}}^{d/2} 2^d.
\end{align}
In view of \eqref{eq:formula-orthogo-mixedproof} and \eqref{eq:formula-Mixed}, it suffices to show that
\begin{equation*}
    \p[\big]{1 + o(1)}\frac{\sqrt{3}}{16}\p[\Big]{\frac{2 d}{\pi n}}^{d/2} 2^d < \p[\big]{1 + o(1)} \delta \cdot 7\cdot 2^{d-6} \p[\Big]{\frac{2d}{\pi n}}^{d/2}.
\end{equation*}
As this is indeed true for any $\delta$ satisfying $1 > \delta > 4\sqrt{3}/7$, we are done.
\end{proof}

\section{Discussion and open problems}
\label{sec:conclusion}

In our work, we have shown that \Cref{conj:erdos} of Erdős continues to give rise to a rich array of intriguing phenomena that remain poorly understood.

Recall that for a set of vectors $V = \set{v_1, \dotsb, v_n} \subseteq \RR^d$, we denote by $\sigma_V$ the random variable $\sigma_V \defined \eps_1 v_1 + \dotsb + \eps_n v_n$ where $\eps_1, \dotsc, \eps_n$ are independent Rademacher random variables.
Further, it is convenient to denote
\begin{equation*}
    F_{d,r}(n) \defined \inf_{V \in (\SS^{d-1})^n} \prob[\big]{\norm{\sigma_V}_2 \leq r}.
\end{equation*}
We now present several natural questions that arise from our work.

\begin{question}
\label{qu:critical2}
When $n$ is restricted to be odd, roughly how fast does
\begin{equation*}
    \inf_{\set{v_1, \dotsc, v_n} \subseteq \SS^1} \prob[\big]{ \norm{\eps_1 v_1 + \cdots + \eps_n v_n}_2 \leq 1 }
\end{equation*}
decay with $n$? Does it decay polynomially in $n$, exponentially in $n$, or in another way?
\end{question}

In other words, \Cref{qu:critical2} asks for the behaviour of $F_{2,1}(n)$ as $n$ goes to infinity while being odd.
Our result in \Cref{thm:lower-bound} establishes a lower bound of $\Omega(0.525^n)$.
As previously mentioned, Gregory Sorkin~\cite{Sorkin25} recently resolved this question by showing that $F_{2,1}(n)$ for $n$ odd indeed exhibits exponential decay, with an upper bound of $O(\sqrt{2}^{n})$.
While our approach in \Cref{thm:lower-bound} leaves room for improvement, as noted in \Cref{rmk:elliptic}, we cannot match the upper bound obtained by Sorkin at this moment.
The precise asymptotics of $F_{2,1}(n)$ remain elusive for $n$ odd, and we leave it an open problem to determine whether
\begin{equation*}
    \lim_{\substack{n\to \infty\\n \not\equiv 0 \mod{2}}} \p[\big]{F_{2,1}(n)}^{1/n}
\end{equation*}
exists, and if so, what its exact value is.

Although the original \Cref{conj:erdos} of Erdős is false when $n$ is odd, we have shown in \Cref{thm:approximate} that an approximate version holds.
Equivalently, \Cref{thm:approximate} shows that, for each $\delta > 0$, the quantity
\begin{equation}
\label{eq:Defcdelta}
    c_\delta \defined \liminf_{\substack{n\to \infty\\n \not\equiv 0 \mod{2}}} n \, F_{2,1+\delta}(n)
\end{equation}
is strictly positive.
However, from the proof of \Cref{thm:approximate}, it is clear that the lower bound on $c_\delta$ we obtain depends very poorly on $\delta$.
More precisely, we obtain $c_\delta = \Omega(\delta^2 e^{-1/\delta^2})$ as $\delta$ approaches $0$, where we have made explicit the lower bound \eqref{eq:hjns} in \Cref{prop:hjns} by carefully tracking the dependencies in~\cite{He2024-cp}.
A natural open question is to understand how $c_\delta$ varies as $\delta > 0$ approaches $0$.
In other words, to determine the behaviour of $F_{2,1+\delta}(n)/n$ for $n$ odd and large, as a function of $\delta > 0$.

Still in the two-dimensional case, we now consider the case where $r = \sqrt{2}$ and $n$ may be even or odd.
As we have seen in the introduction, several conjectures were made under the suspicion that $\prob[\big]{ \norm{\sigma_V}_2 \leq \sqrt{2}}$ is minimised when the vectors $V = \set{v_1, \dotsc, v_n}$ are selected from an orthogonal basis.
However, with \Cref{thm:BetterConstant}, we have shown that a simplicial configuration achieves a lower probability than the orthogonal arrangement.
Despite this improvement, it remains unclear whether this new construction is optimal.

\begin{question}
\label{qu:mercedes}
Which choice of unit vectors $v_1, \dotsc, v_n \in \RR^2$ minimise
\begin{equation*}
    \prob[\big]{ \norm{\eps_1 v_1 + \cdots + \eps_n v_n}_2 \leq \sqrt{2}} \, ?
\end{equation*}
\end{question}

We see no strong evidence suggesting that a simplicial configuration is optimal; in fact, \Cref{thm:OrthoBetterThanSimplex} may even serve as evidence to the contrary.
Identifying good conjectural constructions or even simply gathering evidence in favour or against a certain configuration in two dimension is of great interest.
We note that the answer to \Cref{qu:mercedes} may depend of the parity of $n$.

\Cref{qu:mercedes} can be viewed as a particular case of \Cref{qu:HJNSFunctionf} by He, Ju\v{s}kevi\v{c}ius, Narayanan, and Spiro, where they ask for the behaviour of
\begin{equation*}
    \liminf_{n \to \infty} n \, F_{2,r}(n)
\end{equation*}
as a function of $r$.
A natural extension to higher dimensions is to consider the quantity $n^{d/2} F_{d,r}(n)$.
Furthermore, in view of the apparent importance of the parity, we propose the following refined general problem.

\begin{problem}
\label{eq:fullproblem}
For all $d \geq 2$ and $r \geq 0$, determine the values of
\begin{align*}
    f_{d,r}^{0} \defined \;\;\liminf_{\mathclap{\substack{n\to \infty\\n \equiv d \mod{2}}}}\; n^{d/2} F_{d,r}(n), \quad \text{and } \quad f_{d,r}^{1} \defined \;\;\liminf_{\mathclap{\substack{n\to \infty\\n \not\equiv d \mod{2}}}}\;  n^{d/2} F_{d,r}(n).
\end{align*}
\end{problem}

We believe that obtaining a full description of $f_{d,r}^{0}$ and $f_{d,r}^{1}$ is an incredibly challenging and ambitious endeavour.
Nonetheless, even partial progress in specific cases would be of great interest, and we highlight several such instances where further investigation would be particularly desirable.
It is also convenient to consider the quantity
\begin{equation*}
    f_{d,r} \defined \min \set{f_{d,r}^0, f_{d,r}^1}.
\end{equation*}
Note that $c_\delta$ in \eqref{eq:Defcdelta} is simply $f_{2,1+\delta}^1$ and that \Cref{qu:mercedes} asks what are the vector configurations that attain $f_{2,\sqrt{2}}$.
Since $f_{d,r} = 0$ for all $r < \sqrt{d}$ and $f_{d,\sqrt{d}} > 0$, we believe that the following question is quite natural.

\begin{question}
\label{qu:constant_d}
How does $f_{d,\sqrt{d}}$ varies with $d$?
In other words, what is the maximum constant $C_d$ for which we have
\begin{equation*}
    \prob[\big]{ \norm{\eps_1 v_1 + \cdots + \eps_n v_n}_2 \leq \sqrt{d}} \geq \frac{\p[\big]{C_d - o(1)}}{n^{d/2}}
\end{equation*}
for every choice of unit vectors $v_1, \dotsc, v_n \in \RR^d$, as $n$ grows?
\end{question}

Beck~\cite{Beck1983-ef} has shown in his proof of \Cref{thm:beck} a double exponential lower bound on $f_{d,\sqrt{d}}$ and noted that it would be worthwhile to improve on this estimate.

In higher dimensions, another key problem is to determine whether a double-jump phase transition occurs or not.
Recall that $r_c^{\ast}(d)$ be defined as in \eqref{eq:rcrit-def} to be the infimum of the reals $r > 0$ satisfying
\begin{equation*}
    \liminf_{\substack{n\to \infty\\n \not\equiv d \mod{2}}} F_{d,r}(n) > 0.
\end{equation*}

\begin{question}
\label{qu:double-jump}
For which $d \geq 3$ it is the case that
\begin{equation*}
    \inf_{\set{v_1, \dotsc, v_n} \subseteq \SS^{d-1}} \prob[\big]{ \norm{\eps_1 v_1 + \cdots + \eps_n v_n}_2 \leq r_c^{\ast}(d)}
\end{equation*}
is simultaneously positive and $o(n^{d/2})$ as $n \neq d \mod{2}$ goes to infinity?
\end{question}

In essence, \Cref{qu:double-jump} asks for the weakest statement that demonstrates that the behaviour at the radius $r_c^{\ast}(d)$ is more complex than a single jump from $0$ to $\Theta(n^{d/2})$.
Indeed, a positive answer to \Cref{qu:double-jump} for some $d$ is necessary for a double-jump to occur at $r_c^{\ast}(d)$, but different behaviour is in principle possible and it would be rather interesting if it exists.
If a double-jump indeed occurs, determining the precise behaviour at $r_c^{\ast}(d)$ is the next natural question.
In particular, it would be interesting to determine whether an exponential separation at the double-jump also occurs in higher dimensions, as in the two-dimensional case.

Regardless of whether a double-jump takes place or not, determining the location of $r_c^{\ast}(d)$ is still of independent interest.
Again, it is tempting to believe that $r_c^{\ast}(d) = \sqrt{d - 1}$ following the pattern observed in two dimension.
This would indeed be the case if we have a positive answer to \Cref{qu:vector-balancing}, posed in the introduction and repeated below for emphasis.

\refinedbal*

Nevertheless, \Cref{qu:vector-balancing} remains open and would provide a natural extension of the classical vector balancing results in \cite{Barany1981-mi, Beck1983-ef, Sevast-yanov1980-jf, Spencer1981-qa} from the 1980's.


\section*{Acknowledgements}
\label{sec:ack}

The authors would like to thank Gregory Sorkin for valuable discussions on this problem, as well as providing an improvement on the upper bound from \Cref{thm:radius1}.
The authors are also grateful to Béla Bollobás for his continued support.

The first author is funded by the Internal Graduate Studentship of Trinity College, Cambridge.
The second author is funded by the Department of Pure Mathematics and Mathematical Statistics (DPMMS) of the University of Cambridge.
The third author is partially supported by ERC Starting Grant 101163189 and UKRI Future Leaders Fellowship MR/X023583/1.


\printbibliography


\appendix
\markboth{L. HOLLOM, J. PORTIER, AND V. SOUZA}{DOUBLE-JUMP PHASE TRANSITION FOR THE REVERSE LITTLEWOOD--OFFORD PROBLEM}


\section{Sums of products of binomial coefficients}
\label{app:binom}

In this appendix, we give proofs to \Cref{prop:perturbed-franel-2,,prop:lower-bound-perturbed-franel,,prop:lower-bound-prod-binomial}.

\perturbedfranel*
\begin{proof}[Proof of \Cref{prop:perturbed-franel-2}]
Let $\beta \defined 1/2 + \eps$ and $P \defined m_1 \dotsb m_q$, and define
\begin{align*}
    f(k) &\defined \prod_{i=1}^{q} \binom{m_i}{(m_i+x_i)/2 + k}, \\
    R &\defined 1/m_1 + \dotsb + 1/m_q, \;\text{ and } \\
    G  &\defined 2^{m_1 + \dotsb + m_q} \p[\big]{2/\pi}^{(q-1)/2} \p[\big]{RP}^{-1/2}.
\end{align*}
We are going to split the sum $\sum_{k} f(k)$ into the main contribution, coming from terms with $\abs{k} \leq \ell^\beta$, and an error term, corresponding with terms with $\abs{k} > \ell^\beta$.

We first take care of the main contribution.
For that, we are going to estimate $\binom{m_i}{t}$ when $t = (m_i + x_i)/2 + k$ with $k = O(\ell^{\beta})$.
In this case, we have
\begin{equation*}
    t = m_i/2 + O(k) = (1 + O(k/m_i)) m_i/2,
\end{equation*}
and we may note that the above approximation holds for $m_i - t$ as well.
Thus
\begin{equation*}
\sqrt{\frac{m_i}{2\pi t(m_i-t)}} = \sqrt{\frac{m_i}{2 \pi (m_i/2)^2 \p[\big]{1 + O(k/m_i)} }} = \p[\big]{1 + O(k/m_i)} \sqrt{\frac{2}{\pi m_i}}
\end{equation*}
Observe that the entropy function satisfies $H(1/2 + \eps) = 1 - 2 \eps^2 / \log 2 + O(\eps^3)$, and so
\begin{align*}
    H\p[big]{t/m_i} = H\p[\Big]{\frac{1}{2} + \frac{2k + x_i}{2m_i}} &= 1 - \frac{2}{\log 2}\p[\Big]{\frac{2k+x_i}{2m_i}}^2 + O\p[\big]{((k +x_i)/m_i)^3} \\
    &= 1 - \frac{2 k^2}{m_i^2 \log 2} + O\p[\big]{kX/m_i^2},
\end{align*}
which in particular gives
\begin{equation*}
    2^{m_i H(t/m_i)} = 2^{m_i} e^{-2k^2/m_i + O(kX/m_i)} = \p[\big]{1 + O(kX/m_i)} 2^{m_i} e^{-2k^2/m_i}.
\end{equation*}
But note that \Cref{fact:stirling} gives
\begin{align*}
     \binom{m_i}{(m_i+x_i)/2 + k} = \p[\big]{1 + O(kX/m_i)}  \sqrt{\frac{2}{\pi m_i}} 2^{m_i} e^{- 2k^2/m_i}.
\end{align*}
Therefore, as long as $k = O(\ell^{\beta})$, we have
\begin{align*}
    f(k) &= \p[\big]{1 + O(kX/\ell)} \prod_{i=1}^{q}  \sqrt{\frac{2}{\pi m_i}} 2^{m_i} e^{- 2k^2/m_i} \\
    &= \p[\big]{1 + O(kX/\ell)} 2^{m_1 + \dotsb + m_q} \p[\big]{2/\pi}^{q/2} \p[\big]{P}^{-1/2} \exp\p[\big]{- 2 R k^2}.
\end{align*}
Summing over all $k$ with $\abs{k} \leq \ell^\beta$, we get
\begin{align*}
    \sum_{\abs{k} \leq \ell^{\beta}} f(k) &= 2^{m_1 \dotsb + m_q} \p[\big]{2/\pi}^{q/2} \p[\big]{P}^{-1/2} \sum_{\abs{k} \leq \ell^{3/4}} \p[\big]{1 + O(kX/\ell)} \exp\p[\big]{- 2 R k^2}.
\end{align*}
Note further that we have
\begin{equation*}
    \sum_{\abs{k} \leq \ell^{\beta}} \p[\big]{1 + O(kX/\ell)} \exp\p[\big]{- 2 R k^2} = \p[\big]{1 + O(X/\ell^{1 - \beta})} \sum_{\mathclap{\abs{k} \leq \ell^{\beta}}} \exp\p[\big]{- 2 R k^2}.
\end{equation*}
Finally, as the function $k \mapsto \exp(-2R k^2)$ can be split into two monotone intervals, and $\exp(-2R k^2) =O(1)$, a simple comparison with the integral gives
\begin{align*}
    \sum_{\abs{k} \leq \ell^{\beta}} \exp\p[\big]{- 2 R k^2} = \int_{-\ell^{\beta}}^{\ell^{\beta}} \exp\p[\big]{- 2 R x^2} \dx + O(1).
\end{align*}
Using the following tail inequality\footnote{This can be obtained by Markov's inequality and the fact that $\expec{e^{-\lambda \cN(0,\sigma^2)}} = e^{\sigma^2 \lambda^2/2}$. Indeed, let $\lambda = t/\sigma^2$ and note that $\prob[\big]{\cN(0,\sigma^2) \leq - t} = \prob[\big]{e^{- \lambda \cN(0,\sigma^2)} \geq e^{\lambda t}} \leq e^{\sigma^2 \lambda^2 / 2} \cdot e^{-\lambda t} = e^{-t^2/2\sigma^2}$. } $\prob[\big]{\cN(0,\sigma^2) \leq - t} \leq e^{-t^2/2\sigma^2}$, we have
\begin{align*}
    \int_{-\ell^{\beta}}^{\ell^{\beta}} \exp\p[\big]{- 2 R x^2} \dx
    &= \sqrt{\frac{\pi}{2 R}} - 2\int_{-\infty}^{-\ell^{\beta}} \exp\p[\big]{- 2 R x^2} \dx \\
    &= \p[\big]{1 + O(e^{-2 R \ell^{2\beta}})} \sqrt{\frac{\pi}{2 R}}.
\end{align*}
As $1/\ell \leq R \leq q/\ell$, we have $R \ell^{2\beta} = \Theta(\ell^{2\eps})$ and since $e^{-x} = O_b(1/x^b)$ for any $b \geq 1$, we have
\begin{align*}
    \sum_{\abs{k} \leq \ell^{\beta}} \exp\p[\big]{- 2 R k^2}
    = \p[\big]{1 + O_\eps(1/\ell)}\sqrt{\frac{\pi}{2 R}}.
\end{align*}
Therefore, the main to the sum of $f(k)$ gives
\begin{align}
\label{eq:Contri-Main-Term}
    \sum_{\abs{k}\leq \ell^{\beta}} f(k)
    &= \p[\big]{1 + O(X/\ell^{1-\beta}) + O_\eps(1/\ell)} 2^{m_1 + \dotsb + m_q} \p[\big]{2/\pi}^{q/2} \p[\big]{P}^{-1/2} \sqrt{\frac{\pi}{2R}} \nonumber \\
    &= \p[\big]{1 + O_\eps(X/\ell^{1/2 - \eps})} G.
\end{align}
We now turn to the contribution of the terms with $k > \ell^{\beta}$.
Note that for every $\card{k} > \ell^{\beta}$, we have
\begin{align*}
    f(k) \leq f(\ell^{\beta}) &= \p[\big]{1 + O(X/\ell^{1 - \beta})} 2^{m_1 + \dotsb + m_q} \p[\big]{2/\pi}^{q/2} \p[\big]{P}^{-1/2} \exp\p[\big]{- 2 R \ell^{2\beta}} \\
    &= \p[\big]{1 + O(X/\ell^{1 - \beta})} G \p[\big]{2R/\pi}^{1/2} \exp\p[\big]{- 2 R \ell^{2\beta}} = O_\eps(G/\ell^{2}).
\end{align*}
Finally, as there are $O(\ell)$ terms in the sum, we have
\begin{equation}
\label{eq:Contri-Error}
    \sum_{\abs{k} > \ell^{\beta}} f(k) = O_\eps\p[\big]{\ell f(\ell^\beta)} = O_\eps(G/\ell).
\end{equation}
Combining \eqref{eq:Contri-Main-Term} and \eqref{eq:Contri-Error} finishes the proof.
\end{proof}

We will use similar ideas employed in the proof above to prove \Cref{prop:lower-bound-perturbed-franel,prop:lower-bound-prod-binomial}.
The following elementary fact will also be useful.

\begin{restatable}{proposition}{allmequal}
\label{prop:all-m-equal}
Let $q$ be a fixed positive integer, and let $y_1, \dotsc, y_{q}$ and $n$ be positive real numbers.
If $y_1 + \dotsb + y_{q} = n$, then we have $\p[\big]{\prod_{i=1}^{q} y_i} \p[\big]{ \sum_{i=1}^{q} y_i^{-1}} \leq n^{q-1}/q^{q-2}$,
and the maximum is attained when all $y_i$ are $n/q$.
\end{restatable}
\begin{proof}
We will show that, if $y_i < y_j$ and $0 < x < y_j - y_i$, then replacing $y_i$ by $y_i + x$ and $y_j$ by $y_j - x$ increases $\p[\big]{\prod_{i=1}^{q} y_i} \p[\big]{ \sum_{i=1}^{q} y_i^{-1}}$.
With this result in hand, the claim follows by iteratively performing such replacements to make each $y_i$ equal to $n/q$.

If we fix $i$ and $j$ as above, and define $S \defined \sum_{k \notin \set{i,j}} y_k^{-1}$, then it suffices to prove that
\begin{equation*}
    (y_i + x)(y_j - x)(S + (y_i + x)^{-1} + (y_j - x)^{-1}) > y_i y_j (S + y_i^{-1} + y_j^{-1}),
\end{equation*}
which after expanding and rearranging, it is shown to be equivalent to
\begin{equation*}
    (y_j x - y_i x - x^2)S > 0,
\end{equation*}
which follows immediately from the fact that $(y_i + x)(y_j - x) > y_i y_j$, as required.
\end{proof}

We now proceed to the remaining proofs.

\lbperturbedfranel*
\begin{proof}
Throughout this proof, all asymptotic notation will hold for $n\to\infty$.
Without loss of generality, we may assume $m_1 \leq m_2 \leq \dotsb \leq m_q$.
Fix a sufficiently slowly growing function $g(n)$; in fact, $g(n) = \loglog n$ will suffice.
Suppose first that $m_1 \geq g(n)$.
By \Cref{prop:perturbed-franel-2}, we have
\begin{align*}
    \sum_{k \in \ZZ} \, \prod_{i=1}^q \binom{m_i}{(m_i+x_i)/2 + k} &= \p[\big]{1 + o(1)} 2^{n} \sqrt{\frac{\p{2/\pi}^{q-1}}{ \p[\big]{\prod_{i=1}^{q} m_i} \p[\big]{ \sum_{i=1}^{q} m_i^{-1}}}}.
\end{align*}
Applying \Cref{prop:all-m-equal}, it follows that
\begin{align*}
    \sum_{k \in \ZZ} \, \prod_{i=1}^{q} \binom{m_i}{(m_i+x_i)/2 + k} &\geq 2^n \p[\big]{1 + o(1)} \p[\bigg]{\frac{2}{\pi n}}^{(q-1)/2} q^{(q-2)/2} \\
    &\geq 2^{n-1} \p[\bigg]{\frac{2}{\pi n}}^{(q-1)/2} q^{(q-2)/2},
\end{align*}
for $n$ large enough, as desired.

We now turn to the case $m_1 \leq g(n)$.
We split this case into subcases.

\begin{description}[leftmargin=0em]
\item[Case (i)] Suppose that $m_2 \leq g(n)^3$.
Let $t \geq 2$ be such that $m_1, \dotsc, m_t \leq g(n)^3$ and $m_{t+1}, \dotsc, m_{q} > g(n)^3$.
For every $i \leq t$ we have $\binom{m_i}{(m_i+x_i)/2} \geq 1$ since $m_i \geq \abs{x_i}$.
On the other hand, for every $i > t$, we have by \Cref{fact:stirling} that
\begin{equation*}
    \binom{m_i}{(m_i+x_i)/2} \geq \frac{2^{m_i}}{\sqrt{\pi m_i}}.
\end{equation*}
Therefore, we obtain that
\begin{align*}
    \sum_{k \in \ZZ} \, \prod_{i=1}^{q} \binom{m_i}{(m_i+x_i)/2 + k} &\geq \prod_{i=1}^{q} \binom{m_i}{(m_i+x_i)/2 } \geq  \prod_{i=t+1}^{q} \frac{2^{m_i}}{\sqrt{\pi m_i}}  \\
    &\geq 2^{n-q \, g(n)^3} \pi^{-(q-t)/2} \prod_{i=t+1}^{q} m_i ^{-1/2}.
\end{align*}
As the above product is minimised when all $m_i$ are as equal as possible, given that their sum is fixed, and that $q$ is constant, we have that for $n$ large enough
\begin{align*}
    \sum_{k \in \ZZ} \, \prod_{i=1}^{q} \binom{m_i}{(m_i+x_i)/2 + k} &\geq 2^{n - q\,g(n)^3} \pi^{-(q-t)/2} \p[\bigg]{ \frac{n - q\,g(n)^3}{q-t} }^{-(q-t)/2} \\
    &= \Omega_q \p[\big]{2^n n^{-q/2 + 3/4}}
    \geq 2^{n-1} \p[\bigg]{\frac{2}{\pi n}}^{(q-1)/2} q^{(q-2)/2} \, n^{1/8}
\end{align*}
for $n$ sufficiently large, noting that the second line follows from the first as $g(n) < (\log n / (4q))^{1/3}$.
This concludes Case (i).
Note that we have obtained the extra factor of $n^{1/8}$ not only when $m_1,m_2 \leq C$, but under the weaker assumption that $m_1 \leq g(n)$  and $m_2 \leq g(n)^3$.

\item[Case (ii)]
Suppose now that $m_2 \geq g(n)^3$.
We have that
\begin{align*}
    \sum_{k \in \ZZ} \, \prod_{i=1}^{q} \binom{m_i}{(m_i+x_i)/2 + k} &= \sum_{k =-m_1}^{m_1} \, \prod_{i=1}^{q} \binom{m_i}{(m_i+x_i)/2 + k}.
\end{align*}
For every $k \leq \abs{k}$, we have $m_i \geq \abs{k}^3$ for all $i \geq 2$, so by \Cref{fact:stirling} we have
\begin{equation*}
    \binom{m_i}{(m_i + x_i)/2 + k} = \p[\big]{1 + o(1)}\frac{2^{m_i+1/2}}{\sqrt{\pi m_i}},
\end{equation*}
so finally, we obtain
\begin{align*}
    \sum_{k =-m_1}^{m_1} \, \prod_{i=1}^{q} \binom{m_i}{(m_i+x_i)/2 + k} &= 2^n\p[\big]{1 + o(1)} \prod_{i=2}^{q} \p[\Big]{\frac{2}{\pi m_i}}^{1/2} \geq 2^n\p[\big]{1 + o(1)} \p[\Big]{\frac{2(q-1)}{\pi (n-m_1)}}^{(q-1)/2},
\end{align*}
where the final inequality comes from noting that the expression is minimised when all $m_i$ are equal.
Furthermore, note that $(q-1)^{(q-1)} > q^{q-2}$, we see that the above lower bound is greater than the desired bound, finishing the proof in this Case (ii) too. \qedhere
\end{description}
\end{proof}

Finally, we proceed to the last result we need.

\lbprodbinomial*
\begin{proof}
If $q = 1$, the result is classical and is just an application of Stirling's approximation, so we assume $q \geq 2$.
Assume $m_1 \leq \dotsb \leq m_{q}$.
Fix a sufficiently slowly growing function $g(n)$; such as $g(n) = \loglog n$.
Let $t \geq 0$ be such that $m_1, \dotsc, m_t \leq g(n)$ and $m_{t+1}, \dotsc, m_{q} > g(n)$.
For every $i \leq t$ we have $\binom{m_i}{(m_i+x_i)/2} \geq 1$ and for every $i > t$ we have by \Cref{fact:stirling} that
\begin{equation*}
    \binom{m_i}{(m_i+x_i)/2} \geq \p[\big]{1 + o(1)} 2^{m_i}\sqrt{\frac{2}{\pi m_i}}.
\end{equation*}
Therefore
\begin{align*}
    \prod_{i=1}^{q} \binom{m_i}{(m_i+x_i)/2 } &\geq \p[\big]{1 + o(1)} \prod_{i=t+1}^{q} 2^{m_i} \sqrt{\frac{2}{\pi m_i}}  \\
    &\geq \p[\big]{1 + o(1)} 2^{n-t g(n)} \p[\bigg]{ \frac{2}{\pi}}^{(q-t)/2} \prod_{i=t+1}^{q} m_i ^{-1/2} \\
    &\geq \p[\big]{1 + o(1)} 2^{n-t g(n)} \p[\bigg]{\frac{2(q-t)}{\pi n}}^{(q-t)/2},
\end{align*}
so we are done if $t = 0$.
When $t \geq 1$, we have $g(n) < \log n / (4q)$, so we also done since
\begin{equation*}
    \frac{1}{2^n} \prod_{i=1}^{q} \binom{m_i}{(m_i+x_i)/2 } = \Omega\p[\big]{n^{-q/2+1/4}}. \qedhere
\end{equation*}
\end{proof}

\section{Optimal orthogonal constructions}
\label{app:orthogonal}

Recall that $\cS_h$ is the set of tuples $(x_1, \dotsc, x_d) \in \ZZ^d$ such that $x_1^2 + \dotsb + x_d^2 \leq d$ and that $x_i \equiv h_i \mod{2}$ for all $i$.
We have defined
\begin{equation*}
    f_0(d) \defined \;\; \min_{\mathclap{\substack{h \in \set{0,1}^d\\h_1 + \dotsb + h_d \equiv d (2)}}} \quad \card{\cS_h}
    \quad \text{and } \quad
    f_1(d) \defined \;\; \min_{\mathclap{\substack{h \in \set{0,1}^d\\h_1 + \dotsb + h_d \not\equiv d (2)}}} \quad \card{\cS_h}.
\end{equation*}
We now determine the values of $f_0(d)$ and $f_1(d)$ for all $d \geq 1$.

\begin{proposition}
\label{prop:optimal-orthogonal}
The values of $f_0(d)$ and $f_1(d)$ are given by
\begin{align*}
    f_0(d) = \begin{cases}
        2^d & \text{if $d = 1$}, \\
        2^{d-2} & \text{if $2 \leq d \leq 5$}, \\
        13 \cdot 2^{d-6} & \text{if $6 \leq d \leq 9$ or $d \geq 17$}, \\
        (191+d) \cdot 2^{d-10} & \text{if $10 \leq d \leq 17$}.
    \end{cases}
    &&
    f_1(d) = \begin{cases}
        2^{d-1} & \text{if $1 \leq d \leq 2$}, \\
        2^{d-3} & \text{if $3 \leq d \leq 6$}, \\
        15 \cdot 2^{d-7} & \text{if $d \geq 7$}.
    \end{cases}
\end{align*}
\end{proposition}
\begin{proof}
For $d \geq t$, denote by $f(t,d)$ the cardinality of $\cS_h$ where $h \in \set{0,1}^d$ consists of $t$ coordinates equal to $0$ and $d - t$ coordinates equal to $1$.
Note that, for every $d$, we have $f_0(d) = \min_{\text{ even }i} f(i,d)$ and $f_1(d)= \min_{\text{odd } i} f(i,d)$.
We also remark the expression of $f(i,d)$ for small values of $i$ (which are easily obtained by inspection):
\begin{align*}
    f(0,d) &= 2^d, &
    f(1,d) &= 2^{d-1}, &
    f(2,d) &= 2^{d-2}, &
    f(3,d) &= 2^{d-3}, \\
    f(4,d) &= 9 \cdot 2^{d-4}, &
    f(5,d) &= 11 \cdot 2^{d-5}, &
    f(6,d) &= 13 \cdot 2^{d-6}, &
    f(7,d) &= 15 \cdot 2^{d-7}.
\end{align*}

Using the values above, we obtain the values of $f_0(d)$ when $d \leq 7$ and $f_1(d)$ when $d \leq 8$.
Next, we determine the value of $f_0(d)$ for each $d \geq 8$.
Using the same method, we will then find $f_1(d)$ for each $d \geq 9$.
We start with the following claim.

\begin{claim}
\label{cl:bound-ftd}
For every $d$ and even $t$ satisfying $8 \leq t \leq d$ and $t \neq 10,14$, we have $f(t,d) > f(6,d)$.
\end{claim}
\begin{proof}
First, observe that $f(t,d)/2^d$ is non-decreasing in $d$.
This follow from the fact that by appending a coordinate $1$ to $h$, the number of solutions in $\cS_h$ at least doubles, as the new coordinate can be $\pm 1$.
Therefore, if we have $f(k,t) > f(6,t)$ for some positive integers $k \leq t$, then we also have $f(k,d) > f(6,d)$ for all $d \geq t$.
Indeed, we have
\begin{equation*}
    \frac{f(k,d)}{2^d} \geq \frac{f(k,t)}{2^t} > \frac{f(6,t)}{2^t} = \frac{f(6,t)2^{d-t}}{2^d} = \frac{f(6,d)}{2^d}.
\end{equation*}
Therefore, to prove our claim, it suffices to show that for every even $t \geq 8$, we have $f(t,t) > f(6,t)$.
By definition, $f(t,t)$ is the number of integer solutions to $x_1^2 + \dotsb + x_t^2 \leq t$ where each $x_i$ is even.
Counting only solutions where the variables $x_i$ have values in $\set{-2,0,2}$, we have
\begin{equation*}
    f(t,t) \geq F(t) \defined \sum_{i = 0}^{\floor{t/4}} 2^i \binom{t}{i}.
\end{equation*}
First, we show that if $t \geq 168$, then we have $F(t) > f(6,t)$.
We use the standard inequality that $\binom{n}{k} \geq 2^{n H(k/n)}/(n+1)$ (see for instance Cover and Thomas~\cite[Example 11.1.3]{Cover2005-xs}) where $H$ is the binary entropy function.
Furthermore, note that $1/4 + H(1/4) \geq 1.06$, so we have
\begin{align*}
    F(t) &\geq  2^{\floor{t/4}} \binom{4\floor{t/4}}{\floor{t/4}}
    \geq \frac{2^{\floor{t/4} + 4H(1/4)\floor{t/4}}}{4\floor{t/4}+1}
    \geq \frac{2^{(1/4 + H(1/4))(t - 4)}}{t+1}
    \geq \frac{2^{1.06t}}{2^5(t+1)},
\end{align*}
so it suffices to show that
\begin{align*}
    \frac{2^{1.06t}}{2^5(t+1)} > \frac{13 \cdot 2^t}{2^6} \quad \Leftrightarrow \quad 2^{0.06 t} > \frac{13(t+1)}{2},
\end{align*}
which holds for $t \geq 169$.
It can be easily checked numerically that $F(t) > f(6,t)$ holds for all $8 \leq t \leq 168$ such that $t \neq 10,14$.
Therefore, $f(t,t) \geq F(t) > f(6,t)$ for every even $t \geq 8$ such that $t \neq 10,14$, which finishes the proof.
\end{proof}

From \Cref{cl:bound-ftd}, it follows that for every $d \geq 14$, we have $f_0(d) = \min\set{f(6,d), f(10,d), f(14,d)}$, and for every $10 \leq d \leq 13$, we have $f_0(d) = \min\set{f(6,d), f(10,d)}$.
But note that
\begin{align*}
    f(10,d) = (191 + d)2^{d-10},\hspace{4mm } \text{and} \hspace{4mm }f(14,d) = (2899 + 29d)2^{d-14},
\end{align*}
so $f(10,d) \geq f(6,d)$ for $d \geq 17$ and $f(14,d) > f(6,d)$ for $d \geq 15$, which gives $f_0(d)=f(6,d)$ for $d \geq 17$.
By inspection in the range $10 \leq d \leq 16$, we find that $f_0(d)=f(10,d)$.
Finally, when $8 \leq d \leq 9$, we have $f_0(d) = f(6,d)$, and therefore we have derived the value of $f_0(d)$ for every $d$.

We now determine the value of $f_1(d)$ for each $d \geq 9$ via the same technique.

\begin{claim}
\label{cl:bound-ftd2}
For every $d$ and odd $t$ satisfying $9 \leq t \leq d$, we have $f(t,d) > f(7,d)$.
\end{claim}
\begin{proof}
As in the proof of \Cref{cl:bound-ftd}, it suffices to show that for every odd $t \geq 9$, we have $f(t,t) > f(7,t)$.
We have also proven in \Cref{cl:bound-ftd} that $f(t,t) \geq F(t) \defined \sum_{i = 0}^{\floor{t/4}}2^i \binom{t}{i}$ and that $F(t) > f(6,t)$ for every $t \geq 169$.
Since $f(6,t) > f(7,t)$ for all $t$, is is sufficient to show that $F(t) > f(7,t)$ for all $9 \leq t \leq 168$, which again can be easily checked numerically.
\end{proof}

From \Cref{cl:bound-ftd2}, it follows that for every $d \geq 9$, we have $f_1(d)=f(7,d)$. Therefore we have determined the value of $f_1(d)$ for every $d$.
\end{proof}

\end{document}